\documentclass{amsart}
\usepackage{amssymb,graphicx}

\theoremstyle{definition}
\newtheorem{definition}{Definition}
\theoremstyle{plain}
\newtheorem{theorem}[definition]{Theorem}
\newtheorem{corollary}[definition]{Corollary}
\newtheorem{lemma}[definition]{Lemma}
\newtheorem*{claim}{Claim}

\newcommand{\Area}{\mathop{\mathrm{Area}}}

\begin{document}

\author{Yuya~Koda}
\address{Department of Mathematics, Hiyoshi Campus, Keio University, 4-1-1, Hiyoshi, Kohoku, Yokohama, 223-8521, Japan/ International Institute for Sustainability with Knotted Chiral Meta Matter (WPI-SKCM${}^\text{2}$), Hiroshima University, 1-3-1 Kagamiyama, Higashi-Hiroshima, 739-8526, Japan}
\email{koda@keio.jp}

\author{Kazuto~Takao}
\address{Graduate School of Information Sciences, Tohoku University, 6-3-09 Aoba, Aramaki-aza Aoba-ku, Sendai, 980-8579, Japan}
\email{kazuto.takao.e1@tohoku.ac.jp}

\title[Diagrammatic criteria for strong irreducibility]{Diagrammatic criteria for strong irreducibility of Heegaard splittings and finiteness of Goeritz groups}
\thanks{The first author is supported by JSPS KAKENHI Grant Numbers JP20K03588, JP21H00978, JP23H05437.
The second author was supported by JSPS KAKENHI Grant Number JP18K13409.}

\begin{abstract}
We give two criteria for diagrams of Heegaard splittings of $3$-manifolds.
Weaker one of them guarantees that the splitting is strongly irreducible, and the stronger one guarantees in addition that the Goeritz group is finite.
They are generalizations of the criteria introduced by Casson--Gordon and Lustig--Moriah.
Their original ones require as input the diagram formed by a pair of maximal disk systems.
Our present ones accept as input that of arbitrary disk systems, including minimal disk systems.
In other words, our criteria accept Heegaard diagrams.
\end{abstract}

\maketitle

\section{Introduction}

A Heegaard splitting is a decomposition of a $3$-manifold into two handlebodies.
Here, a {\it handlebody} is an orientable $3$-manifold obtained by attaching some $1$-handles to a $3$-ball.
We define a {\it Heegaard splitting} of a $3$-manifold $M$ as the triple $\left( \Sigma ,H,H^\ast \right) $ of a closed surface $\Sigma $ and two handlebodies $H$ and $H^\ast $ in $M$ such that $H\cup H^\ast =M$ and $\Sigma =H\cap H^\ast =\partial H=\partial H^\ast $.
The {\it genus} of $\left( \Sigma ,H,H^\ast \right) $ is the genus of $\Sigma $.
It is an important fact that every closed, orientable, connected $3$-manifold admits a Heegaard splitting.
Since Heegaard splittings of genera zero and one are well understood, we restrict our attention to those of genera greater than one.

A Heegaard diagram is a standard diagrammatic representation of a $3$-manifold and a Heegaard splitting of it.
Let $\left( \Sigma ,H,H^\ast \right) $ be a Heegaard splitting of genus $g$ of a $3$-manifold.
An {\it essential} disk of a handlebody, say $H$, is a disk $D$ properly embedded in $H$ such that $\partial D$ bounds no disk in $\partial H$.
A {\it disk system} of $H$ is a family ${\mathcal D}$ of essential disks of $H$ that are mutually disjoint and mutually non-parallel, and cut $H$ into some $3$-balls.
Note that the cardinality of ${\mathcal D}$ is at least $g$ and at most $3g-3$.
The disk system ${\mathcal D}$ is said to be {\it minimal} if the cardinality attains $g$, and {\it maximal} if $3g-3$.
Let ${\mathcal D}^\ast $ be a disk system of $H^\ast $.
The boundary curves of the disks in ${\mathcal D}$ and ${\mathcal D}^\ast $ form a diagram on the surface $\Sigma $, which determines the manifold and the splitting.
It is called a {\it Heegaard diagram} if ${\mathcal D}$ and ${\mathcal D}^\ast $ are minimal.
Non-minimal disk systems might also be useful under some circumstances, but are excessive for representing the manifold and the splitting.

\subsection{Strong irreducibility}

It is a-priori hard to extract geometric information of the manifold from a Heegaard splitting, but has been done from the strong irreducibility of it.
A Heegaard splitting $\left( \Sigma ,H,H^\ast \right) $ of a $3$-manifold $M$ is said to be {\it strongly irreducibile} if any essential disk of $H$ and any essential disk of $H^\ast $ have non-empty intersection.
If $\left( \Sigma ,H,H^\ast \right) $ is strongly irreducible, then it is irreducible as suggested by the name, and then $M$ is irreducible as shown by Haken \cite{Haken}.
If $\left( \Sigma ,H,H^\ast \right) $ is irreducible but not strongly irreducible, then $M$ has an essential surface as shown by Casson--Gordon \cite{Casson-Gordon}.
If $\left( \Sigma ,H,H^\ast \right) $ is strongly irreducible, then $M$ may have only essential surfaces that the component submanifolds are glued along with bounded complexities, as shown in \cite{Lackenby,Li07,Li10,Souto}.

We remark that the idea of strong irreducibility streamed abundant studies of so-called Hempel distance.
It is an integral measure of complexity for Heegaard splittings, introduced by Hempel \cite{Hempel}.
The distance is greater than zero if and only if the splitting is irreducible, and greater than one if and only if strongly irreducible.
The distance bounds from below the genera of essential surfaces in the $3$-manifold, as shown in \cite{Hartshorn}.

Casson--Gordon introduced a diagrammatic criterion, called the rectangle condition, for a Heegaard splitting to be strongly irreducible.
It requires as input the diagram formed by a pair of maximal disk systems of the hadlebodies.

We generalize the rectangle condition to diagrams including Heegaard diagrams.
\begin{theorem}\label{strongly_irreducible}
A Heegaard splitting is strongly irreducible if the rectangle condition holds for the diagram formed by a pair of disk systems of the handlebodies.
\end{theorem}
\noindent
See Section~\ref{definition} for definitions, and Subsection~\ref{proof_irreducibility} for a proof.
See also \cite[Example~2.7]{Lustig-Moriah} for an example of a Heegaard splitting with infinitely many different pairs of disk systems that form diagrams satisfying the rectangle condition.

\subsection{Finiteness of Goeritz groups}

The {\it Goeritz group} of a Heegaard splitting is the group of symmetries of the splitting.
More precisely, it is the group of isotopy classes of orientation-preserving homeomorphisms of the $3$-manifold preserving each of the handlebodies setwise.
It is known that the Goeritz group is finite if the Hempel distance is greater than three, and infinite if the distance is less than two as shown by Johnson \cite{Johnson10}.
In the case where the distance is equal to three, Zou \cite{Zou} recently announced a proof of the finiteness.
In the case where the distance is equal to two, there exist Heegaard splittings with finite and infinite Goeritz groups as shown in \cite{Johnson(11)} and \cite{Iguchi-Koda}, respectively.
So far, we do not know the nature that separates these two phenomena.
Note that measuring the distance of a given splitting is another difficult problem.

Lustig--Moriah \cite{Lustig-Moriah} introduced a diagrammatic criterion, called the double rectangle condition, for a Goeritz group to be finite.
In fact, they did not mention Goeritz group, but their conclusion implies its finiteness.
The condition is based on maximal disk systems as well as Casson--Gordon's.

We generalize the double rectangle condition to diagrams including Heegaard diagrams.
The key theorem is the following.
\begin{theorem}\label{finiteness}
For any Heegaard splitting of any $3$-manifold, there exist at most finitely many pairs of disk systems of the hadlebodies that form diagrams satisfying the double rectangle condition, up to isotopies in the respective hadlebodies.
\end{theorem}
\noindent
See Section~\ref{definition} for definitions, and Subsection~\ref{proof_finiteness} for a proof.
See also Section~\ref{example} for a family of examples satisfying the double rectangle condition.
We also observe that smaller disk systems may work better than maximal disk systems, and that the Heegaard splittings in the examples have Hempel distance two.
As a consequence of Theorem~\ref{finiteness}, we have the following.
\begin{corollary}\label{goeritz_finiteness}
The Goeritz group of a Heegaard splitting is finite if the double rectangle condition holds for the diagram formed by a pair of disk systems of the handlebodies.
\end{corollary}
\begin{proof}
Let $\left( \Sigma ,H,H^\ast \right) $ denote the Heegaard splitting.
Let $X$ denote the quotient, by isotopies in $H$ and $H^\ast $, of the set of pairs of disk systems of $H$ and $H^\ast $ that form diagrams satisfying the double rectangle condition.
It is non-empty by the hypothesis, and finite by Theorem~\ref{finiteness}.
Let $f$ be a representative of an element of the Goeritz group, that is, an orientation-preserving homeomorphism of the $3$-manifold preserving $H$ and $H^\ast $ setwise.
The map $f$ induces a permutation of the finite set $X$.
For some integer $a$, the induced permutation by $f^a$ preserves an element of $X$.
Let $\left( {\mathcal D},{\mathcal D}^\ast \right) $ be a representative of that element.
After isotopies in $H$ and $H^\ast $, we may suppose that $f^a$ preserves the union of the disks in ${\mathcal D}$ and ${\mathcal D}^\ast $ setwise.
By the definition of the double rectangle condition, the disks cut off the surface $\Sigma $ into finitely many polygonal disks.
The map $f^a$ induces a permutation of the polygonal disks.
For some integer $b$, the map $f^{ab}$ preserves each polygonal disk setwise.
It follows that $f^{ab}$ is isotopic to the identity in each polygonal disk, and hence in $\Sigma $.
By \cite[Theorem~3.7]{Fomenko-Matveev}, the map $f^{ab}$ is isotopic to the identity in $H$ and $H^\ast $, and hence in the entire manifold.
Thus the Goeritz group is a torsion group.
Note that the Goeritz group can be naturally thought of as a subgroup of the mapping class group of $\Sigma $.
By \cite[Theorem~6.9]{Farb-Margalit} (see also \cite{Serre}), the mapping class group has a representation into a finite group such that the kernel is torsion-free.
It follows that the torsion subgroup is finite.
\end{proof}

\section{Definitions}\label{definition}

In this section, we give preliminary definitions and the generalized definitions of rectangle condition and double rectangle condition.

We begin with miscellaneous preliminary definitions.
Let $\Sigma $ be a closed, orientable surface.
Let ${\mathcal C}$ be a family of mutually disjoint, essential, simple, closed curves on $\Sigma $, and ${\mathcal C}^\ast $ be another such family.
We say that ${\mathcal C}$ and ${\mathcal C}^\ast $ {\it intersect essentially} in $\Sigma $ if
\begin{itemize}
\item the curves in ${\mathcal C}$ and ${\mathcal C}^\ast $ intersect transversely, and
\item no bigon in $\Sigma $ is cobounded by subarcs of curves in ${\mathcal C}$ and ${\mathcal C}^\ast $.
\end{itemize}
Suppose that ${\mathcal C}$ and ${\mathcal C}^\ast $ intersect essentially in $\Sigma $, and that the curves are oriented.
We regard an {\it orientation} of a curve on $\Sigma $ as one of the two ways of labeling the two sides of the curve with the signs.
A rectangle in $\Sigma $ is said to be of {\it type} $\left( \left( C_1,\varepsilon _1\right) ,\left( C_2,\varepsilon _2\right) ;\left( C^\ast _1,\nu _1\right) ,\left( C^\ast _2,\nu _2\right) \right) $ if
\begin{itemize}
\item $C_1,C_2\in {\mathcal C}$ and $C^\ast _1,C^\ast _2\in {\mathcal C}^\ast $ and $\varepsilon _1,\varepsilon _2,\nu _1,\nu _2\in \{ -,+\} $,
\item its interior is disjoint from the curves in ${\mathcal C}$ and ${\mathcal C}^\ast $, and
\item it is cobounded by subarcs of $C_1$, $C_2$, $C^\ast _1$, $C^\ast _2$ as in the left of Figure~\ref{fig_rectangle}.
\end{itemize}
A rectangle in $\Sigma $ is said to be of {\it type} $\left( \left( C_1,\varepsilon _1\right) ,C_2,\left( C_3,\varepsilon _3\right) ;\left( C^\ast _1,\nu _1\right) ,\left( C^\ast _2,\nu _2\right) \right) $ if
\begin{itemize}
\item $C_1,C_2,C_3\in {\mathcal C}$ and $C^\ast _1,C^\ast _2\in {\mathcal C}^\ast $ and $\varepsilon _1,\varepsilon _3,\nu _1,\nu _2\in \{ -,+\} $, and
\item it is composed of two rectangles of types $\left( \left( C_1,\varepsilon _1\right) ,\left( C_2,-\right) ;\left( C^\ast _1,\nu _1\right) ,\left( C^\ast _2,\nu _2\right) \right) $ and $\left( \left( C_2,+\right) ,\left( C_3,\varepsilon _3\right) ;\left( C^\ast _1,\nu _1\right) ,\left( C^\ast _2,\nu _2\right) \right) $ as in the right of Figure~\ref{fig_rectangle}.
\end{itemize}
Let $H$ be a handlebody.
Let $D_1,D_2,\ldots ,D_n$ be mutually disjoint, essential disks of $H$, and ${\mathcal D}$ denote the family $\left\{ D_1,D_2,\ldots ,D_n\right\} $.
By convention, we denote the union $D_1\cup D_2\cup \cdots \cup D_n$ by $\overline{\mathcal D}$, and the family $\left\{ \partial D_1,\partial D_2,\ldots ,\partial D_n\right\} $ by $\partial {\mathcal D}$.
Let ${\mathcal E}$ be another family of mutually disjoint, essential disks of $H$.
The families ${\mathcal D}$ and ${\mathcal E}$ are said to be {\it isotopic} in $H$ if $\overline{\mathcal D}$ and $\overline{\mathcal E}$ are isotopic in $H$.
We say that ${\mathcal D}$ and ${\mathcal E}$ {\it intersect essentially} in $H$ if
\begin{itemize}
\item the disks in ${\mathcal D}$ and ${\mathcal E}$ intersect transversely,
\item $\partial {\mathcal D}$ and $\partial {\mathcal E}$ intersect essentially in $\partial H$, and
\item no lens in $H$ is cobounded by round subdisks of $\overline{\mathcal D}$ and $\overline{\mathcal E}$.
\end{itemize}
Note that if ${\mathcal D}$ and ${\mathcal E}$ intersect essentially in $H$, then there is no closed component in $\overline{\mathcal D}\cap \overline{\mathcal E}$, by the irreducibility of $H$.
For a set $A$ and two subsets $A_-$ and $A_+$ of it, we call the pair $\left( A_-,A_+\right) $ a {\it partition} of $A$ if $A_-\cup A_+=A$ and $A_-\cap A_+=\emptyset $.
By a {\it graph} simply, we mean a finite, $1$-dimensional, simplicial complex without loop edges or multiple edges.
A graph is said to be {\it $2$-connected} if it is connected even after deleting an arbitrary vertex.
Let ${\mathcal G}$ be a graph whose vertex set is $A$, and suppose that $\left( A_-,A_+\right) $ is a partition of $A$.
We say that ${\mathcal G}$ is {\it doubly $2$-connected} with respect to $\left( A_-,A_+\right) $ if ${\mathcal G}$ is connected even after deleting an arbitrary vertex in $A_-$ and an arbitrary vertex in $A_+$.
For example, the graph in Figure~\ref{fig_graph} is not doubly $2$-connected with respect to $\left( \{ 1,2,3\} ,\{ 4,5,6\} \right) $ because it is disconnected after deleting the vertices $2$ and $5$, while it is doubly $2$-connected with respect to $\left( \{ 1,2,4,5\} ,\{ 3,6\} \right) $.

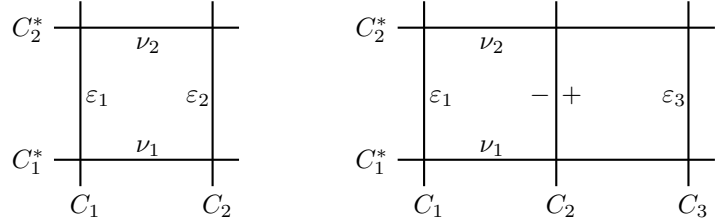
\begin{figure}[ht]
\begin{picture}(270,80)(0,0)
\thicklines
\put(20,70){\line(1,0){70}}
\put(20,20){\line(1,0){70}}
\put(30,10){\line(0,1){70}}
\put(80,10){\line(0,1){70}}
\put(4,67){$C^\ast _2$}
\put(4,17){$C^\ast _1$}
\put(26,0){$C_1$}
\put(76,0){$C_2$}
\put(51,62){$\nu _2$}
\put(51,23){$\nu _1$}
\put(32,42){$\varepsilon _1$}
\put(70,42){$\varepsilon _2$}
\put(150,70){\line(1,0){120}}
\put(150,20){\line(1,0){120}}
\put(160,10){\line(0,1){70}}
\put(210,10){\line(0,1){70}}
\put(260,10){\line(0,1){70}}
\put(134,67){$C^\ast _2$}
\put(134,17){$C^\ast _1$}
\put(156,0){$C_1$}
\put(206,0){$C_2$}
\put(256,0){$C_3$}
\put(181,62){$\nu _2$}
\put(181,23){$\nu _1$}
\put(162,42){$\varepsilon _1$}
\put(200,42){$-$}
\put(212,42){$+$}
\put(250,42){$\varepsilon _3$}
\end{picture}
\caption{Rectangles.}
\label{fig_rectangle}
\end{figure}

\begin{figure}[ht]
\begin{picture}(85,55)(0,0)
\put(0,0){\includegraphics{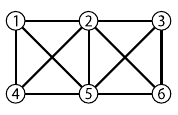}}
\end{picture}
\caption{A graph.}
\label{fig_graph}
\end{figure}

\subsection{Criteria for minimal disk systems}\label{criteria_minimal}

We give definitions of the criteria in the case of minimal disk systems first for simplicity.
Let $g$ be an integer greater than one, and $\left( \Sigma ,H,H^\ast \right) $ be a Heegaard splitting of genus $g$ of a closed, orientable, connected $3$-manifold.
Let ${\mathcal D}$ and ${\mathcal D}^\ast $ be minimal disk systems of $H$ and $H^\ast $, respectively, such that $\partial {\mathcal D}$ and $\partial {\mathcal D}^\ast $ intersect essentially in $\Sigma $.
Let $D_1,D_2,\ldots ,D_g$ denote the disks in ${\mathcal D}$, and $D^\ast _1,D^\ast _2,\ldots ,D^\ast _g$ denote the disks in ${\mathcal D}^\ast $.
Endow the boundary curves of these disks with orientations.
For $(i,\varepsilon ),(j,\delta )\in \{ 1,2,\ldots ,g\} \times \{ -,+\} $, let ${\mathcal G}_{i,\varepsilon ,j,\delta }$ denote the graph defined as:
\begin{itemize}
\item The vertex set is $\{ 1,2,\ldots ,g\} \times \{ -,+\} $;
\item Two vertices $(v,\nu )$, $(w,\mu )$ span an edge if there is a rectangle on $\Sigma $ of type $\left( \left( \partial D_i,\varepsilon \right) ,\left( \partial D_j,\delta \right) ;\left( \partial D^\ast _v,\nu \right) ,\left( \partial D^\ast _w,\mu \right) \right) $.
\end{itemize}
Let ${\mathcal G}$ denote the graph defined as:
\begin{itemize}
\item The vertex set is $\{ 1,2,\ldots ,g\} \times \{ -,+\} $;
\item Two vertices $(i,\varepsilon )$, $(j,\delta )$ span an edge if ${\mathcal G}_{i,\varepsilon ,j,\delta }$ is $2$-connected.
\end{itemize}
\begin{definition}
We say that $\left( {\mathcal D},{\mathcal D}^\ast \right) $, or the diagram formed by $\left( {\mathcal D},{\mathcal D}^\ast \right) $, satisfies the {\it rectangle condition} if ${\mathcal G}$ is $2$-connected.
\end{definition}
For $d\in \{ 1,2,\ldots ,g\} $ and $\kappa \in \{ -,+\} $, let $\varLambda _{d,\kappa }$ denote the set $\left( \{ 1,2,\ldots ,g\} \times \{ -,+\} \right) \setminus \left\{ (d,\kappa)\right\} $.
For $d\in \{ 1,2,\ldots ,g\} $ and $\left( i_-,\varepsilon _-\right) \in \varLambda _{d,-}$ and $\left( i_+,\varepsilon _+\right) \in \varLambda _{d,+}$, let ${\mathcal H}_{d,i_-,\varepsilon _-,i_+,\varepsilon _+}$ denote the graph defined as:
\begin{itemize}
\item The vertex set is $\{ 1,2,\ldots ,g\} \times \{ -,+\} $;
\item Two vertices $(v,\nu )$, $(w,\mu )$ span an edge if there is a rectangle on $\Sigma $ of type $\left( \left( \partial D_{i_-},\varepsilon _-\right) ,\partial D_d,\left( \partial D_{i_+},\varepsilon _+\right) ;\left( \partial D^\ast _v,\nu \right) ,\left( \partial D^\ast _w,\mu \right) \right) $.
\end{itemize}
For $d\in \{ 1,2,\ldots ,g\} $, let ${\mathcal H}_d$ denote the graph defined as:
\begin{itemize}
\item The vertex set is $\left( \{ -\} \times \varLambda _{d,-}\right) \cup \left( \{ +\} \times \varLambda _{d,+}\right) $;
\item Two vertices $(\kappa ,i,\varepsilon )$, $(\kappa ,j,\delta )$ span an edge if ${\mathcal G}_{i,\varepsilon ,j,\delta }$ is $2$-connected;
\item Two vertices $\left( -,i_-,\varepsilon _-\right) $, $\left( +,i_+,\varepsilon _+\right) $ span an edge if ${\mathcal H}_{d,i_-,\varepsilon _-,i_+,\varepsilon _+}$ is $2$-connected.
\end{itemize}
\begin{definition}\label{def_drc}
We say that $\left( {\mathcal D},{\mathcal D}^\ast \right) $, or the diagram formed by $\left( {\mathcal D},{\mathcal D}^\ast \right) $, satisfies the {\it double rectangle condition} if for every $d\in \{ 1,2,\ldots ,g\} $, the graph ${\mathcal H}_d$ is doubly $2$-connected with respect to $\left( \{ -\} \times \varLambda _{d,-},\{ +\} \times \varLambda _{d,+}\right) $, and the symmetric condition switching $H$ and $H^\ast $ holds.
\end{definition}
Note that these criteria does not essentially depend on orientations of the curves.

\subsection{Criteria for general disk systems}\label{criteria_general}

We now give the general definitions of the criteria.
These include both those in the previous subsection and those by Casson--Gordon and Lustig--Moriah \cite{Lustig-Moriah}.
Let $\left( \Sigma ,H,H^\ast \right) $ be a Heegaard splitting of genus greater than one, of a closed, orientable, connected $3$-manifold.
Let ${\mathcal D}$ and ${\mathcal D}^\ast $ be disk systems of $H$ and $H^\ast $, respectively, such that $\partial {\mathcal D}$ and $\partial {\mathcal D}^\ast $ intersect essentially in $\Sigma $.
Let $D_1,D_2,\ldots ,D_n$ denote the disks in ${\mathcal D}$, and $D^\ast _1,D^\ast _2,\ldots ,D^\ast _{n^\ast }$ denote the disks in ${\mathcal D}^\ast $.
Endow the boundary curves of these disks with orientations.
Let $H_1,H_2,\ldots ,H_m$ denote the complementary components of $\overline{\mathcal D}$ in $H$, and $H^\ast _1,H^\ast _2,\ldots ,H^\ast _{m^\ast }$ denote those of $\overline{{\mathcal D}^\ast }$ in $H^\ast $.
For each $k\in \{ 1,2,\ldots ,m\} $, let $A_k$ denote the subset of $\{ 1,2,\ldots ,n\} \times \{ -,+\} $ such that $(i,\varepsilon )$ belongs to $A_k$ if $\partial H_k$ has a path converging to $\partial D_i$ from the $\varepsilon $-side.
For example, $A_1$ is the set $\left\{ (1,-),(2,+),(3,-),(3,+)\right\} $ if $H_1$ appears as in Figure~\ref{fig_handle}.
We similarly define $A^\ast _\ell$ for each $\ell \in \left\{ 1,2,\ldots ,m^\ast \right\} $.
For $k\in \{ 1,2,\ldots ,m\} $ and $\ell \in \{ 1,2,\ldots ,m^\ast \} $ and $(i,\varepsilon ),(j,\delta )\in A_k$, let ${\mathcal G}_{k,\ell ,i,\varepsilon ,j,\delta }$ denote the graph defined as:
\begin{itemize}
\item The vertex set is $A^\ast _\ell $;
\item Two vertices $(v,\nu )$, $(w,\mu )$ span an edge if there is a rectangle on $\Sigma $ of type $\left( \left( \partial D_i,\varepsilon \right) ,\left( \partial D_j,\delta \right) ;\left( \partial D^\ast _v,\nu \right) ,\left( \partial D^\ast _w,\mu \right) \right) $.
\end{itemize}
For $k\in \{ 1,2,\ldots ,m\} $, let ${\mathcal G}_k$ denote the graph defined as:
\begin{itemize}
\item The vertex set is $A_k$;
\item Two vertices $(i,\varepsilon )$, $(j,\delta )$ span an edge if ${\mathcal G}_{k,\ell ,i,\varepsilon ,j,\delta }$ is $2$-connected for every $\ell \in \{ 1,2,\ldots ,m^\ast \} $.
\end{itemize}
\begin{definition}
We say that $\left( {\mathcal D},{\mathcal D}^\ast \right) $, or the diagram formed by $\left( {\mathcal D},{\mathcal D}^\ast \right) $, satisfies the {\it rectangle condition} if ${\mathcal G}_k$ is $2$-connected for every $k\in \{ 1,2,\ldots ,m\} $.
\end{definition}
For $d\in \{ 1,2,\ldots ,n\} $ and $\kappa \in \{ -,+\} $, let $k_{d,\kappa }$ denote the element in $\{ 1,2,\ldots ,m\} $ such that $(d,\kappa )$ belongs to $A_{k_{d,\kappa }}$, and $\varLambda _{d,\kappa }$ denote the set $A_{k_{d,\kappa }}\setminus \left\{ (d,\kappa )\right\} $.
For $\ell \in \{ 1,2,\ldots ,m^\ast \} $ and $d\in \{ 1,2,\ldots ,n\} $ and $\left( i_-,\varepsilon _-\right) \in \varLambda _{d,-}$ and $\left( i_+,\varepsilon _+\right) \in \varLambda _{d,+}$, let ${\mathcal H}_{\ell ,d,i_-,\varepsilon _-,i_+,\varepsilon _+}$ denote the graph defined as:
\begin{itemize}
\item The vertex set is $A^\ast _\ell $;
\item Two vertices $(v,\nu )$, $(w,\mu )$ span an edge if there is a rectangle on $\Sigma $ of type $\left( \left( \partial D_{i_-},\varepsilon _-\right) ,\partial D_d,\left( \partial D_{i_+},\varepsilon _+\right) ;\left( \partial D^\ast _v,\nu \right) ,\left( \partial D^\ast _w,\mu \right) \right) $.
\end{itemize}
For $d\in \{ 1,2,\ldots ,n\} $, let ${\mathcal H}_d$ denote the graph defined as:
\begin{itemize}
\item The vertex set is $\left( \{ -\} \times \varLambda _{d,-}\right) \cup \left( \{ +\} \times \varLambda _{d,+}\right) $;
\item Two vertices $(\kappa ,i,\varepsilon )$, $(\kappa ,j,\delta )$ span an edge if ${\mathcal G}_{k_{d,\kappa },\ell ,i,\varepsilon ,j,\delta }$ is $2$-connected for every $\ell \in \{ 1,2,\ldots ,m^\ast \} $;
\item Two vertices $\left( -,i_-,\varepsilon _-\right) $, $\left( +,i_+,\varepsilon _+\right) $ span an edge if ${\mathcal H}_{\ell ,d,i_-,\varepsilon _-,i_+,\varepsilon _+}$ is $2$-connected for every $\ell \in \{ 1,2,\ldots ,m^\ast \} $.
\end{itemize}
Then we define the {\it double rectangle condition} exactly as in Definition~\ref{def_drc}.
Note that these criteria does not essentially depend on orientations of the curves.

\begin{figure}[ht]
\begin{picture}(170,130)(0,0)
\put(0,0){\includegraphics{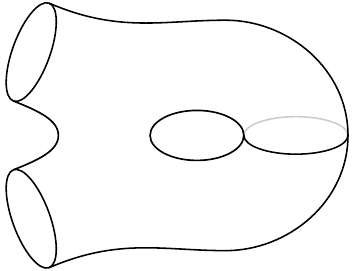}}
\put(60,93){$H_1$}
\put(9,102){$D_1$}
\put(25,96){$-$}
\put(9,22){$D_2$}
\put(25,28){$+$}
\put(137,62){$D_3$}
\end{picture}
\caption{A possible location of $H_1$.}
\label{fig_handle}
\end{figure}
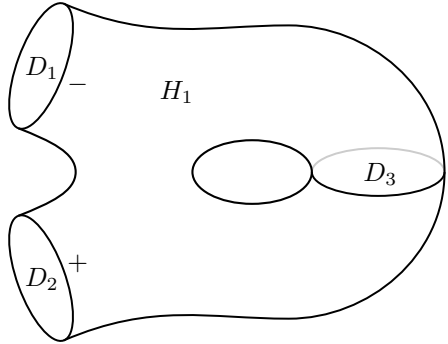

\section{Examples}\label{example}

In this section, we give a family of examples satisfying the rectangle condition and the double rectangle condition.

We construct $3$-manifolds, Heegaard splittings, and pairs of minimal disk systems as follows.
Let $g$ be an integer greater than one, and let $H$ and $H^\ast $ be handlebodies with boundary surfaces of genus $g$.
Let ${\mathcal D}$ and ${\mathcal D}^\ast $ be minimal disk systems of $H$ and $H^\ast $, respectively, and let $D_1,D_2,\ldots ,D_g$ and $D^\ast _1,D^\ast _2,\ldots ,D^\ast _g$ denote their disks, respectively.
Let $h\colon \partial H^\ast \to \partial H$ be a homeomorphism that maps $\partial D_i^\ast $ to $\partial D_i$ for each $i\in \{ 1,2,\ldots ,g\} $.
Let $\gamma $ denote the simple, closed curve on $\partial H$ as in Figure~\ref{fig_curve_1}, and $\tau \colon \partial H\to \partial H$ denote the Dehn twist along $\gamma $.
Let $l$ be an integer such that $|l|\geq 2$, and $M$ denote the $3$-manifold obtained by gluing $H$ and $H^\ast $ by $\tau ^l\circ h$.
Let $\Sigma $ denote the common boundary of $H$ and $H^\ast $ in $M$.
Note that $\left( \Sigma ,H,H^\ast \right) $ is a Heegaard splitting of $M$ of genus $g$.
We restrict to the case where $g=3$ and $l=2$ for simplicity in the following paragraphs, and leave the other cases to the reader.

\begin{figure}[ht]
\begin{picture}(340,125)(0,0)
\put(0,0){\includegraphics{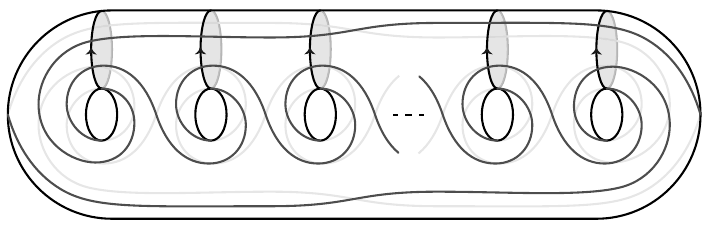}}
\put(44,110){$D_1$}
\put(97,110){$D_2$}
\put(150,110){$D_3$}
\put(190,110){$\cdots $}
\put(232,110){$D_{g-1}$}
\put(287,110){$D_g$}
\put(34,80){{\footnotesize $-$}}
\put(46,80){{\footnotesize $+$}}
\put(87,80){{\footnotesize $-$}}
\put(99,80){{\footnotesize $+$}}
\put(140,80){{\footnotesize $-$}}
\put(152,80){{\footnotesize $+$}}
\put(225,80){{\footnotesize $-$}}
\put(237,80){{\footnotesize $+$}}
\put(278,80){{\footnotesize $-$}}
\put(290,80){{\footnotesize $+$}}
\put(185,22){$\gamma $}
\end{picture}
\caption{Essential disks of $H$, orientations of their boundaries, and a simple, closed curve on $\partial H$.}
\label{fig_curve_1}
\end{figure}
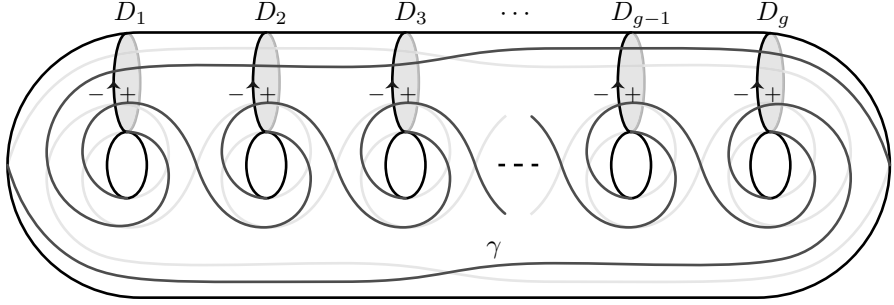

We verify the criteria as follows.
Endow $\partial D_1$, $\partial D_2$, $\partial D_3$ with orientations as in Figure~\ref{fig_curve_1}, and $\partial D_1^\ast $, $\partial D_2^\ast $, $\partial D_3^\ast $ with orientations to agree with those of $\partial D_1$, $\partial D_2$, $\partial D_3$ via $h$.
Then the curves of $\partial {\mathcal D}^\ast $ appear on $\partial H$ as in Figure~\ref{fig_curve_2}.
Let ${\mathcal G}_{i,\varepsilon ,j,\delta }$, ${\mathcal G}$, $\varLambda _{d,\kappa }$, ${\mathcal H}_{d,i_-,\varepsilon _-,i_+,\varepsilon _+}$ and ${\mathcal H}_d$ be as in Subsection~\ref{criteria_minimal}.
Note that, if $\gamma $ has a subarc that leaves $D_i$ to the $\varepsilon $-side and goes without intersecting $\overline{\mathcal D}$ and arrives $D_j$ from the $\delta $-side, then a bunch of subarcs of $\partial \overline{{\mathcal D}^\ast }$ goes the same way, and the $\varepsilon $-side of $\partial D_i^\ast $ faces the $\delta $-side of $\partial D_j^\ast $ in the bunch.
These show that the graph ${\mathcal G}_{i,\varepsilon ,j,\delta }$ contains the subgraph as in the left of Figure~\ref{fig_subgraphs}, which shows that ${\mathcal G}_{i,\varepsilon ,j,\delta }$ is $2$-connected.
Similarly for the graph ${\mathcal H}_{d,i_-,\varepsilon _-,i_+,\varepsilon _+}$.
They show that ${\mathcal G}$ coincides with the graph in the left of Figure~\ref{fig_subgraphs}, and ${\mathcal H}_1$ is as shown in the right.
We can see that ${\mathcal G}$ is $2$-connected, and ${\mathcal H}_1$ is doubly $2$-connected with respect to $\left( \{ -\} \times \varLambda _{1,-},\{ +\} \times \varLambda _{1,+}\right) $.
Similarly for the graphs ${\mathcal H}_2$ and ${\mathcal H}_3$.
We can also see that the symmetric condition switching $H$ and $H^\ast $ holds.
Thus, $\left( {\mathcal D},{\mathcal D}^\ast \right) $ satisfies the rectangle condition and the double rectangle condition.
By Theorem~\ref{strongly_irreducible} and Corollary~\ref{goeritz_finiteness}, the Heegaard splitting $\left( \Sigma ,H,H^\ast \right) $ is strongly irreducible, and its Goeritz group is finite.

\begin{figure}[ht]
\begin{picture}(340,130)(0,0)
\put(0,0){\includegraphics{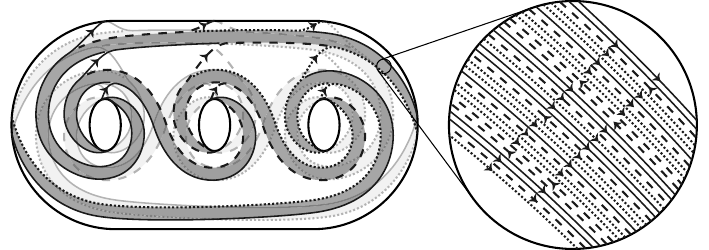}}
\put(42,115){$\partial D_1^\ast $}
\put(95,115){$\partial D_2^\ast $}
\put(148,115){$\partial D_3^\ast $}
\end{picture}
\caption{The three curves of $\partial {\mathcal D}^\ast $ on $\partial H$.}
\label{fig_curve_2}
\end{figure}
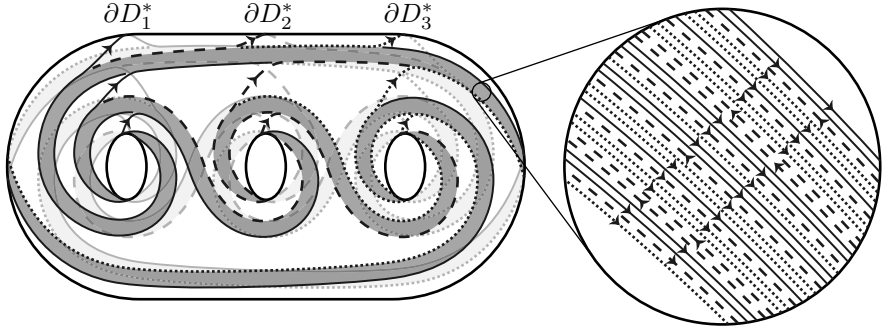

\begin{figure}[ht]
\begin{picture}(210,125)(0,0)
\put(0,30){\includegraphics{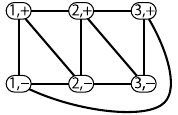}}
\put(120,0){\includegraphics{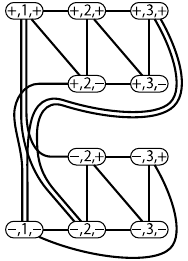}}
\end{picture}
\caption{A subgraph of ${\mathcal G}_{i,\varepsilon ,j,\delta }$ and the graph ${\mathcal H}_1$.}
\label{fig_subgraphs}
\end{figure}

We remark that the double rectangle condition fails for the pairs of maximal disk systems obtained by naively adding disks.
Let $D_4$, $D_5$, $D_6$ denote the essential disks of $H$ as in Figure~\ref{fig_curve_3}, and $D^\ast _4$, $D^\ast _5$, $D^\ast _6$ be essential disks of $H^\ast $ such that $h\left( \partial D_4^\ast \right) =\partial D_4$, $h\left( \partial D_5^\ast \right) =\partial D_5$, $h\left( \partial D_6^\ast \right) =\partial D_6$.
Let ${\mathcal E}$ and ${\mathcal E}^\ast $ denote the maximal disk systems $\left\{ D_1,D_2,\ldots ,D_6\right\} $ of $H$ and $\left\{ D^\ast _1,D^\ast _2,\ldots ,D^\ast _6\right\} $ of $H^\ast $, respectively.
Then the pair $\left( {\mathcal E},{\mathcal E}^\ast \right) $ fails the double rectangle condition, in the absence of rectangles of the type illustrated in dark gray in Figure~\ref{fig_curve_3}.
This supports that our generalization has some advantage over the original by Lustig--Moriah \cite{Lustig-Moriah}.

\begin{figure}[ht]
\begin{picture}(210,130)(0,0)
\put(7,6){\includegraphics{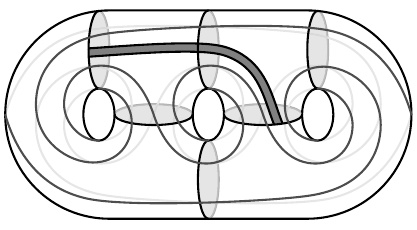}}
\put(0,59){$\gamma $}
\put(51,116){$D_1$}
\put(104,116){$D_2$}
\put(157,116){$D_3$}
\put(71,46){$D_4$}
\put(104,0){$D_5$}
\put(124,46){$D_6$}
\end{picture}
\caption{Additional essential disks of $H$, and an example of missing types of rectangles.}
\label{fig_curve_3}
\end{figure}
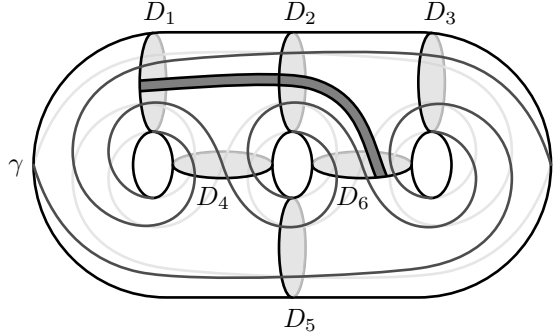

We also remark that their Hempel distance is equal to two.
On one hand, by the rectangle condition, the distance is greater than or equal to two.
On the other hand, by the definition, the sequence from $\partial D_1$ to $\partial D_1^\ast $ through the curve $C$ as in Figure~\ref{fig_curve_4} shows that the distance is less than or equal to two.

\begin{figure}[ht]
\begin{picture}(210,125)(0,0)
\put(7,0){\includegraphics{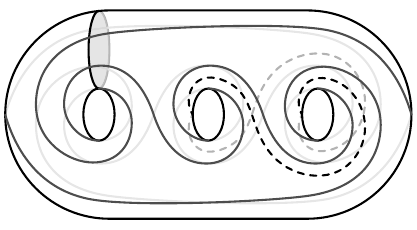}}
\put(51,110){$D_1$}
\put(0,53){$\gamma $}
\put(133,24){$C$}
\end{picture}
\caption{A simple, closed curve on $\Sigma $ disjoint from $\partial D_1$ and $\partial D_1^\ast $.}
\label{fig_curve_4}
\end{figure}
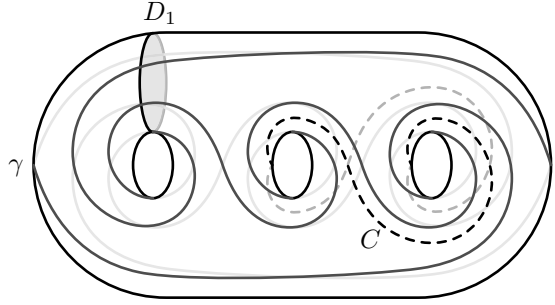

\section{Proofs}

In this section, we give proofs of the main theorems.

The following notions and facts play important roles throughout the proofs.
Let $H$ be a handlebody, and ${\mathcal D}$ be a disk system of $H$.
A {\it wave disk} with respect to ${\mathcal D}$ is a bigon $\Omega $ in $H$ such that
\begin{itemize}
\item $\Omega \cap \partial H$ and $\Omega \cap \overline{\mathcal D}$ are the two edges of the bigon $\Omega $, and
\item no bigon in $\partial H$ is cobounded by $\Omega \cap \partial H$ and any subarc of $\partial \overline{\mathcal D}$.
\end{itemize}
The arc $\Omega \cap \partial H$ is called a {\it wave} with respect to ${\mathcal D}$.
This generalizes the classical notion in the case where ${\mathcal D}$ is maximal.
If $E$ is an essential disk of $H$ such that $\{ E\} $ and ${\mathcal D}$ intersect essentially in $H$, then $E$ is disjoint from $\overline{\mathcal D}$ or contains a wave disk with respect to ${\mathcal D}$.
This can be seen by recalling that there is no closed component in $E\cap \overline{\mathcal D}$, and considering an outermost subdisk of $E$ cut off by the arcs of $E\cap \overline{\mathcal D}$.

We use the following convention.
For a handlebody $H$, and a disk system ${\mathcal D}$ of $H$, we would like to consider ``components cut off from $H$ by $\overline{\mathcal D}$'' and regard them as closed $3$-balls.
Note that the closure of a complementary component of $\overline{\mathcal D}$ in $H$ is, however, possibly a handlebody with one or more $1$-handles as in Figure~\ref{fig_handle}.
To make things clear, we take the universal covering space of $H$, denoted by $\widetilde{H}$, and the covering map, denoted by $\pi _H$.
Let $\widetilde{\mathcal D}$ denote the family of lifts in $\widetilde{H}$ of disks of ${\mathcal D}$.
Let ${\mathcal B}({\mathcal D})$ denote the family of closures of complementary components in $\widetilde{H}$ of the lifted disks.
Note that ${\mathcal B}({\mathcal D})$ consists of infinitely many closed $3$-balls.
For $B\in {\mathcal B}({\mathcal D})$, let $\widetilde{\mathcal D}_B$ denote the family of disks in $\widetilde{\mathcal D}$ contained in $\partial B$.
When the curves of $\partial {\mathcal D}$ are endowed with orientations, we regard the disks of ${\mathcal D}$ and $\widetilde{\mathcal D}$ as endowed with the induced orientations.
A disk $\widetilde{D}$ in $\widetilde{\mathcal D}_B$ is called the $\varepsilon $-{\it side lift} of $\pi _H\bigl( \widetilde{D}\bigr) $ in $B$, if $B$ is adjacent to $\widetilde{D}$ from the $\varepsilon $-side, for $\varepsilon \in \{ -,+\} $.

\subsection{Strong irreducibility}\label{proof_irreducibility}

We prove Theorem~\ref{strongly_irreducible}.
Let $\left( \Sigma ,H,H^\ast \right) $, ${\mathcal D}$, ${\mathcal D}^\ast $ be as in Subsection~\ref{criteria_general}.
Assume that $\left( {\mathcal D},{\mathcal D}^\ast \right) $ satisfies the rectangle condition but $\left( \Sigma ,H,H^\ast \right) $ is not strongly irreducible.
Then there exist an essential disk $E$ of $H$ and an essential disk $E^\ast $ of $H^\ast $ such that $E$ and $E^\ast $ are disjoint.
We can arrange that $\left\{ E\right\} $ and ${\mathcal D}$ intersect essentially in $H$, and similarly for $E^\ast $, after some isotopies if necessary.
Then $E$ is disjoint from $\overline{\mathcal D}$ or contains a wave disk with respect to ${\mathcal D}$, and similarly for $E^\ast $.
By the following lemma, we have a contradiction to the disjointness of $E$ and $E^\ast $.
\begin{lemma}\label{waves_intersect}
Let $\left( \Sigma ,H,H^\ast \right) $, ${\mathcal D}$, ${\mathcal D}^\ast $ be as in Subsection~\ref{criteria_general}.
Let $O$ be either an essential disk of $H$ disjoint from $\overline{\mathcal D}$ or a wave disk with respect to ${\mathcal D}$.
Define $O^\ast $ similarly in $H^\ast $.
If $\left( {\mathcal D},{\mathcal D}^\ast \right) $ satisfies the rectangle condition, then $O$ and $O^\ast $ have non-empty intersection.
\end{lemma}
\begin{proof}
We use the notation in Subsection~\ref{criteria_general}.
Let $k_0$ denote the element in $\{ 1,2,\ldots ,m\} $ such that the closure of $H_{k_0}$ contains $O$, and $\ell _0$ denote that in $\{ 1,2,\ldots ,m^\ast \} $ for $O^\ast $.
Let $B_{k_0}$ denote the closure of a lift of $H_{k_0}$ in $\widetilde{H}$.
Define $B^\ast _{\ell _0}$ similarly over $H^\ast _{\ell _0}$.
Let $\widetilde{O}$ denote the lift of $O$ in $B_{k_0}$, and $\widetilde{O^\ast }$ denote the lift of $O^\ast $ in $B^\ast _{\ell _0}$.
Let $\widetilde{D}_{i,\varepsilon }$ denote the $\varepsilon $-side lift of $D_i$ in $B_{k_0}$ for $(i,\varepsilon )\in A_{k_0}$.
Define $\widetilde{D^\ast }_{v,\nu }$ similarly for $(v,\nu )\in A^\ast _{\ell _0}$.
Let $(I,J)$ denote the partition of $\bigl\{ (i,\varepsilon )\in A_{k_0}\bigm| \widetilde{D}_{i,\varepsilon }\cap \widetilde{O}=\emptyset \bigr\} $ such that $\widetilde{O}$ separates $\bigl\{ \widetilde{D}_{i,\varepsilon }\bigm| (i,\varepsilon )\in I\bigr\} $ and $\bigl\{ \widetilde{D}_{j,\delta }\bigm| (j,\delta )\in J\bigr\} $ in $B_{k_0}$.
Since the graph ${\mathcal G}_{k_0}$ is $2$-connected, there exist a member $\left( i_0,\varepsilon _0\right) $ in $I$ and a member $\left( j_0,\delta _0\right) $ in $J$ that span an edge in ${\mathcal G}_{k_0}$.
By the definition, the graph ${\mathcal G}_{k_0,\ell _0,i_0,\varepsilon _0,j_0,\delta _0}$ is $2$-connected.
Let $(V,W)$ denote the partition of $\bigl\{ (v,\nu )\in A^\ast _{\ell _0}\bigm| \; \widetilde{D^\ast }_{v,\nu }\cap \widetilde{O^\ast }=\emptyset \bigr\} $ such that $\widetilde{O^\ast }$ separates $\bigl\{ \widetilde{D^\ast }_{v,\nu }\bigm| (v,\nu )\in V\bigr\} $ and $\bigl\{ \widetilde{D^\ast }_{w,\mu }\bigm| (w,\mu )\in W\bigr\} $ in $B^\ast _{\ell _0}$.
There exist a member $\left( v_0,\nu _0\right) $ in $V$ and a member $\left( w_0,\mu _0\right) $ in $W$ that span an edge in ${\mathcal G}_{k_0,\ell _0,i_0,\varepsilon _0,j_0,\delta _0}$.
By the definition, there exists a rectangle $R$ on $\Sigma $ of type $\left( \left( \partial D_{i_0},\varepsilon _0\right) ,\left( \partial D_{j_0},\delta _0\right) ;\left( \partial D^\ast _{v_0},\nu _0\right) ,\left( \partial D^\ast _{w_0},\mu _0\right) \right) $, as illustrated in Figure~\ref{fig_intersect_1}.
Since $\widetilde{O}$ separates $\widetilde{D}_{i_0,\varepsilon _0}$ and $\widetilde{D}_{j_0,\delta _0}$ in $B_{k_0}$, the rectangle $R$ contains a subarc of $O\cap \Sigma $ that separates the two edges in $\partial D_{i_0}$ and $\partial D_{j_0}$.
Similarly $R$ also contains a subarc of $O^\ast \cap \Sigma $ that separates the two edges in $\partial D^\ast _{v_0}$ and $\partial D^\ast _{w_0}$.
It follows that $O\cap \Sigma $ and $O^\ast \cap \Sigma $ intersect in $R$.
\end{proof}

\begin{figure}[ht]
\begin{picture}(170,130)(0,0)
\put(0,0){\includegraphics{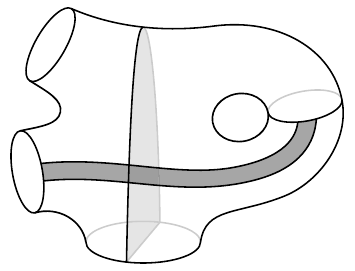}}
\put(66,70){$O$}
\put(6,44){$D_{i_0}$}
\put(142,78){$D_{j_0}$}
\put(35,57){$\partial D^\ast _{v_0}$}
\put(35,34){$\partial D^\ast _{w_0}$}
\put(130,27){\vector(-1,1){19}}
\put(132,20){$R$}
\end{picture}
\caption{The rectangle $R$ in a possible case where $O$ is a wave disk.}
\label{fig_intersect_1}
\end{figure}
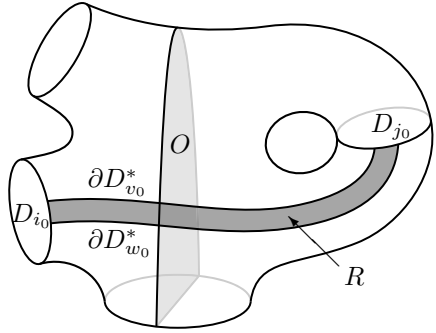

\subsection{Finiteness of disk systems}\label{proof_finiteness}

We prove Theorem~\ref{finiteness} in this subsection.

We have the following relation between our criteria as their names suggest.
\begin{lemma}\label{lem_criteria_implication}
Let ${\mathcal D}$ and ${\mathcal D}^\ast $ be as in Subsection~\ref{criteria_general}.
If $\left( {\mathcal D},{\mathcal D}^\ast \right) $ satisfies the double rectangle condition, then $\left( {\mathcal D},{\mathcal D}^\ast \right) $ satisfies the rectangle condition.
\end{lemma}
\begin{proof}
We use the notation in Subsection~\ref{criteria_general}.
Let $k$ be any element in $\{ 1,2,\ldots ,m\} $.
We consider deleting an arbitrary vertex $(i,\varepsilon )$ from the graph ${\mathcal G}_k$.
Since ${\mathcal G}_k$ has at least three vertices, we can take another vertex $(d,\kappa )$ than $(i,\varepsilon )$.
By the double rectangle condition, the graph ${\mathcal H}_d$ is doubly $2$-connected with respect to $\left( \{ -\} \times \varLambda _{d,-},\{ +\} \times \varLambda _{d,+}\right) $.
Hence, ${\mathcal H}_d$ is connected even after deleting the vertex $(\kappa ,i,\varepsilon )$ and any vertex in $\{ -\kappa \} \times \varLambda _{d,-\kappa }$.
There is a natural simplicial map from ${\mathcal H}_d$ to ${\mathcal G}_k$ whereby $\{ \kappa \} \times \varLambda _{d,\kappa }$ injects to $A_k\setminus \{ (d,\kappa )\} $ and $\{ -\kappa \} \times \varLambda _{d,-\kappa }$ collapses to $(d,\kappa )$.
These show that ${\mathcal G}_k$ is connected even after deleting $(i,\varepsilon )$, to conclude that ${\mathcal G}_k$ is $2$-connected.
\end{proof}

We use the following notation.
Let $g$ be an integer greater than one, and $\left( \Sigma ,H,H^\ast \right) $ be a Heegaard splitting of genus $g$ of a closed, orientable, connected $3$-manifold.
We fix maximal disk systems ${\mathcal E}$ and ${\mathcal E}^\ast $ of $H$ and $H^\ast $, respectively, such that $\partial {\mathcal E}$ and $\partial {\mathcal E}^\ast $ intersect essentially in $\Sigma $.
Let $\mathfrak{DR}$ denote the set of pairs of disk systems of $H$ and $H^\ast $ such that $\left( {\mathcal D},{\mathcal D}^\ast \right) $ belongs to $\mathfrak{DR}$ if
\begin{itemize}
\item $\partial {\mathcal D}$ and $\partial {\mathcal D}^\ast $ intersect essentially in $\Sigma $,
\item $\left( {\mathcal D},{\mathcal D}^\ast \right) $ satisfies the double rectangle condition,
\item ${\mathcal E}$ and ${\mathcal D}$ intersect essentially in $H$,
\item ${\mathcal E}^\ast $ and ${\mathcal D}^\ast $ intersect essentially in $H^\ast $,
\item $\partial {\mathcal E}$ and $\partial {\mathcal D}^\ast $ intersect essentially in $\Sigma $,
\item $\partial {\mathcal E}^\ast $ and $\partial {\mathcal D}$ intersect essentially in $\Sigma $, and
\item $\overline{\mathcal D}\cap \overline{{\mathcal D}^\ast }\cap \overline{\mathcal E}=\overline{\mathcal D}\cap \overline{{\mathcal D}^\ast }\cap \overline{{\mathcal E}^\ast }=\overline{\mathcal D}\cap \overline{\mathcal E}\cap \overline{{\mathcal E}^\ast }=\overline{{\mathcal D}^\ast }\cap \overline{\mathcal E}\cap \overline{{\mathcal E}^\ast }=\emptyset $.
\end{itemize}
We also define ${\mathfrak R}$ by dropping ``double'' in the second condition.
Note that ${\mathfrak R}$ contains $\mathfrak{DR}$ by Lemma~\ref{lem_criteria_implication}.
Note also that if $\left( {\mathcal D},{\mathcal D}^\ast \right) $ satisfies the top two conditions, then we can arrange the remaining conditions by isotopies of ${\mathcal D}$ and ${\mathcal D}^\ast $ in $H$ and $H^\ast $, respectively.
The task in this subsection is, therefore, to show the finiteness of $\mathfrak{DR}$ modulo isotopies in the respective hadlebodies.

First, we consider some $1$-dimensional objects, and show the finiteness of their types.
Let ${\mathcal D}$ and ${\mathcal D}'$ be disk systems of $H$ intersecting essentially in $H$.
An {\it intersection arc} of ${\mathcal D}$ with respect to ${\mathcal D}'$ is a component arc of $\overline{\mathcal D}\cap \overline{{\mathcal D}'}$.
A lift of it in $\widetilde{H}$ is called an {\it intersection arc} of $\widetilde{\mathcal D}$ with respect to $\widetilde{{\mathcal D}'}$.
An {\it isolated circle} of ${\mathcal D}$ with respect to ${\mathcal D}'$ is a component circle of $\partial \overline{\mathcal D}$ disjoint from $\overline{{\mathcal D}'}$.
A {\it connecting arc} of ${\mathcal D}$ with respect to ${\mathcal D}'$ is the closure of a component arc of $\partial \overline{\mathcal D}\setminus \overline{{\mathcal D}'}$.
A lift of it in $\widetilde{H}$ is called a {\it connecting arc} of $\widetilde{\mathcal D}$ with respect to $\widetilde{{\mathcal D}'}$.
Let $\left( {\mathcal D}_1,{\mathcal D}_1^\ast \right) $ and $\left( {\mathcal D}_2,{\mathcal D}_2^\ast \right) $ be pairs in ${\mathfrak R}$.
Intersection arcs $\alpha _1$ and $\alpha _2$ of ${\mathcal D}_1$ and ${\mathcal D}_2$, respectively, with respect to ${\mathcal E}$ are said to be $\left( {\mathcal E},{\mathcal E}^\ast \right) $-{\it parallel} if $\left( \alpha _1,\partial \alpha _1\right) $ and $\left( \alpha _2,\partial \alpha _2\right) $ are isotopic in $\left( \overline{\mathcal E},\partial \overline{\mathcal E}\setminus \overline{{\mathcal E}^\ast }\right) $.
Isolated circles $C_1$ and $C_2$ of ${\mathcal D}_1$ and ${\mathcal D}_2$, respectively, with respect to ${\mathcal E}$ are said to be $\left( {\mathcal E},{\mathcal E}^\ast \right) $-{\it parallel} if $C_1$ and $C_2$ are isotopic in $\Sigma \setminus \overline{\mathcal E}$.
Connecting arcs $\beta _1$ and $\beta _2$ of ${\mathcal D}_1$ and ${\mathcal D}_2$, respectively, with respect to ${\mathcal E}$ are said to be $\left( {\mathcal E},{\mathcal E}^\ast \right) $-{\it parallel} if $\left( \beta _1,\partial \beta _1\right) $ and $\left( \beta _2,\partial \beta _2\right) $ are isotopic in $\left( \Sigma ,\partial \overline{\mathcal E}\setminus \overline{{\mathcal E}^\ast }\right) $.
Let $\mathcal{IA}$, $\mathcal{IC}$, $\mathcal{CA}$ denote the sets
\begin{gather*}
\bigl\{ \, \text{intersection arcs of ${\mathcal D}$ with respect to ${\mathcal E}$}\bigm| \left( {\mathcal D},{\mathcal D}^\ast \right) \in {\mathfrak R}\bigr\} \big/ \text{$\left( {\mathcal E},{\mathcal E}^\ast \right) $-parallel},\\
\bigl\{ \, \text{isolated circles of ${\mathcal D}$ with respect to ${\mathcal E}$}\bigm| \left( {\mathcal D},{\mathcal D}^\ast \right) \in {\mathfrak R}\bigr\} \big/ \text{$\left( {\mathcal E},{\mathcal E}^\ast \right) $-parallel},\\
\bigl\{ \, \text{connecting arcs of ${\mathcal D}$ with respect to ${\mathcal E}$}\bigm| \left( {\mathcal D},{\mathcal D}^\ast \right) \in {\mathfrak R}\bigr\} \big/ \text{$\left( {\mathcal E},{\mathcal E}^\ast \right) $-parallel},
\end{gather*}
respectively.
\begin{lemma}\label{lem_arc_type}
The sets $\mathcal{IA}$, $\mathcal{IC}$, $\mathcal{CA}$ are finite.
\end{lemma}
\begin{proof}
We can see the finiteness of $\mathcal{IA}$ similarly to \cite[Lemma~3.3]{Lustig-Moriah}, and that of $\mathcal{IC}$ more easily.
We show that of $\mathcal{CA}$ as follows, similarly to \cite[Lemma~3.4]{Lustig-Moriah}.
Let $Y_1,Y_2,\ldots ,Y_{2g-2}$ be $3$-balls in ${\mathcal B}\left( {\mathcal E}\right) $ whose images under $\pi _H$ cover the whole handlebody $H$.
Let $\widetilde{E}_{p,1}$, $\widetilde{E}_{p,2}$, $\widetilde{E}_{p,3}$ denote the disks of $\widetilde{\mathcal E}_{Y_p}$ for each $p\in \{ 1,2,\ldots ,2g-2\} $.
We may suppose that ${\mathfrak R}$ is not empty.
It follows by Theorem~\ref{strongly_irreducible} that $\pi _H\bigl( \widetilde{E}_{p,q}\bigr) $ and any disk $E^\ast $ in ${\mathcal E}^\ast $ have non-empty intersection for each $q\in \{ 1,2,3\} $.
The lift in $\partial \widetilde{E}_{p,q}$ of each of the intersection points is adjacent to a closed arc properly embedded in $\partial \widetilde{H}\cap Y_p$ whose image under $\pi _H$ lies in $\partial E^\ast $.
Fix such a closed arc $\widetilde{\theta }_{p,q}$ for $p\in \{ 1,2,\ldots ,2g-2\} $ and $q\in \{ 1,2,3\} $.
Let $s_{p,q}$ denote the {\it twisting number} of $\widetilde{\theta }_{p,q}$ along $\partial \widetilde{E}_{p,q}$, that is, the number of twists relative to a reference arc on $\partial \widetilde{H}\cap Y_p$ and an orientation of $\partial \widetilde{E}_{p,q}$.
The ambiguities in those choices do not affect our conclusion below.
Let $S$ denote the finite set
\begin{equation*}
\bigl\{ s\in {\mathbb Z}\bigm| p\in \{ 1,2,\ldots ,2g-2\} ,\ q\in \{ 1,2,3\} ,\ \left| s-s_{p,q}\right| \leq 2\bigr\} .
\end{equation*}
Now, we take a pair $\left( {\mathcal D},{\mathcal D}^\ast \right) $ in ${\mathfrak R}$ and a connecting arc $\beta $ of ${\mathcal D}$ with respect to ${\mathcal E}$.
Suppose that $p$ is the index in $\{ 1,2,\ldots ,2g-2\} $ such that $\pi _H\left( Y_p\right) $ contains $\beta $.
Let $\widetilde{\beta }$ denote the lift of $\beta $ in $Y_p$.
Suppose that $q$ is one of the indices in $\{ 1,2,3\} $ such that $\widetilde{\beta }$ has an endpoint in $\partial \widetilde{E}_{p,q}$.
Let $t$ denote the twisting number of $\widetilde{\beta }$ along $\partial \widetilde{E}_{p,q}$.
The following claim guarantees that $t$ belongs to $S$.
\begin{claim}
$\left| t-s_{p,q}\right| \leq 2$.
\end{claim}
\begin{proof}
Assume that $\left| t-s_{p,q}\right| >2$.
Note that $\widetilde{E}_{p,q}$ contains one or more intersection arcs of $\widetilde{\mathcal D}$ with respect to $\widetilde{\mathcal E}$, since $\widetilde{\beta }$ has an endpoint there.
Let $\widetilde{\alpha }$ be outermost one of them.
Let $\widetilde{\beta }_1$ and $\widetilde{\beta }_2$ denote the connecting arcs of $\widetilde{\mathcal D}$ with respect to $\widetilde{\mathcal E}$ in $Y_p$ adjacent to $\partial \widetilde{\alpha }$.
Let $t_1$ and $t_2$ denote the twisting numbers of $\widetilde{\beta }_1$ and $\widetilde{\beta }_2$, respectively, along $\partial \widetilde{E}_{p,q}$.
We have either that $\widetilde{\beta }=\widetilde{\beta }_1$ or that $\widetilde{\beta }=\widetilde{\beta }_2$, or that $\widetilde{\beta }$ is disjoint from $\widetilde{\beta }_1$ and $\widetilde{\beta }_2$.
It follows that $\left| t_1-t\right| \leq 1$ and $\left| t_2-t\right| \leq 1$, and hence that $\left| t_1-s_{p,q}\right| >1$ and $\left| t_2-s_{p,q}\right| >1$.
Then $\pi _H\bigl( \widetilde{\theta }_{p,q}\bigr) $ contains a wave $\omega $ with respect to ${\mathcal D}$ as illustrated in Figure~\ref{fig_twists}.
By shifting $\omega $ slightly, we obtain a wave with respect to ${\mathcal D}$ disjoint from $E^\ast $.
This contradicts the rectangle condition by Lemma~\ref{waves_intersect}.
\end{proof}
\noindent
Hence, we obtain an injection from $\mathcal{CA}$ to the finite set
\begin{equation*}
\{ 1,2,\ldots ,2g-2\} \times \pi _0\left( \partial \overline{{\mathcal E}}\setminus \overline{{\mathcal E}^\ast }\right) \times \pi _0\left( \partial \overline{{\mathcal E}}\setminus \overline{{\mathcal E}^\ast }\right) \times S\times S
\end{equation*}
by assigning $\left( p,e,e',t,t'\right) $ to $[\beta ]$ if $\beta $ has the endpoints in $e$ and $e'$, and has a lift in $Y_p$ with twisting numbers $t$ and $t'$, respectively.
\end{proof}

\begin{figure}[ht]
\begin{picture}(260,320)(0,0)
\put(0,120){\includegraphics{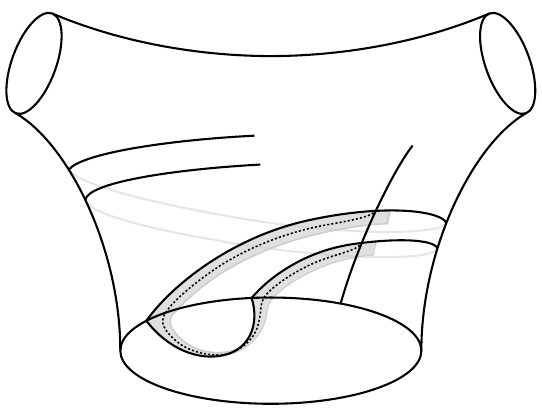}}
\put(60,273){$Y_p$}
\put(126,255){$\widetilde{\beta }_1$}
\put(129,237){$\widetilde{\beta }_2$}
\put(198,254){$\widetilde{\theta }_{p,q}$}
\put(221,211){\vector(-1,0){42}}
\put(225,208){$\widetilde{\omega }$}
\put(94,139){$\widetilde{\alpha }$}
\put(152,146){$\widetilde{E}_{p,q}$}
\put(70,0){\includegraphics{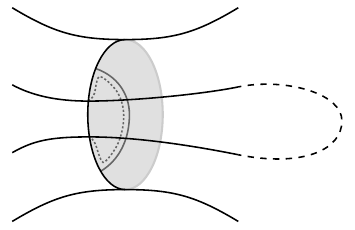}}
\put(65,63){$\widetilde{\theta }_{p,q}$}
\put(129,72){$\widetilde{\alpha }$}
\put(90,74){$\widetilde{\beta }_1$}
\put(102,76){\vector(2,-1){12}}
\put(90,25){$\widetilde{\beta }_2$}
\put(102,30){\vector(2,1){12}}
\put(168,69){$\widetilde{\omega }$}
\end{picture}
\caption{The top shows arcs on $\partial Y_p$ and a subarc $\widetilde{\omega }$ whose image under $\pi _H$ is a wave $\omega $ with respect to ${\mathcal D}$.
The gray band represents a part of a disk in $\widetilde{\mathcal D}$, and the dots represent an arc whose image and $\omega $ cobound a wave disk.
The bottom shows them so that the disk in $\widetilde{\mathcal D}$ looks plain.}
\label{fig_twists}
\end{figure}
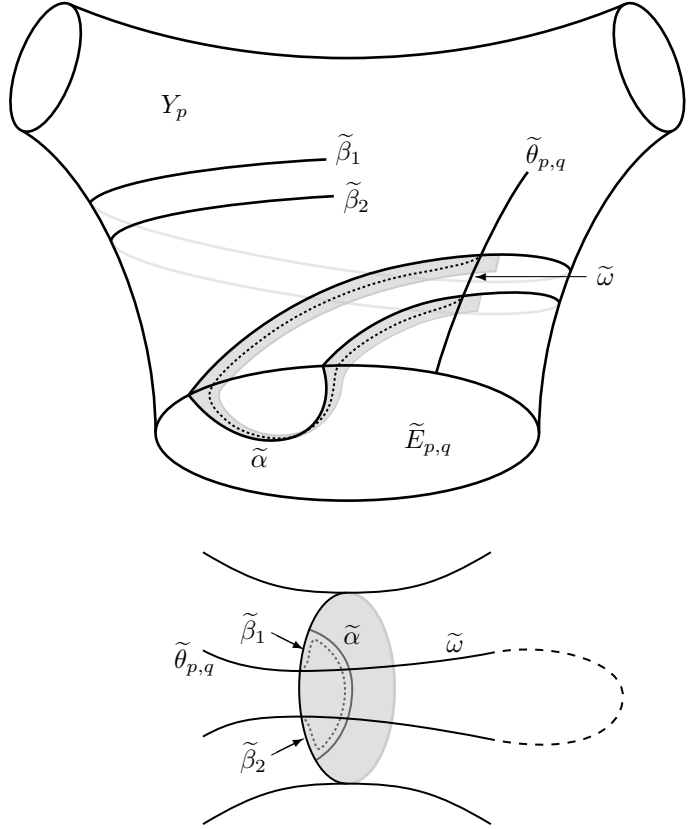

Second, we consider some $2$-dimensional objects, and show the finiteness of their types.
For disk systems ${\mathcal D}$ and ${\mathcal D}'$ of $H$ intersecting essentially in $H$, a {\it disk piece} of ${\mathcal D}$ with respect to ${\mathcal D}'$ is the closure of a component of $\overline{\mathcal D}\setminus \overline{{\mathcal D}'}$.
A lift of it in $\widetilde{H}$ is called a {\it disk piece} of $\widetilde{\mathcal D}$ with respect to $\widetilde{{\mathcal D}'}$.
For pairs $\left( {\mathcal D}_1,{\mathcal D}_1^\ast \right) $ and $\left( {\mathcal D}_2,{\mathcal D}_2^\ast \right) $ in ${\mathfrak R}$, disk pieces $\Delta _1$ and $\Delta _2$ of ${\mathcal D}_1$ and ${\mathcal D}_2$, respectively, with respect to ${\mathcal E}$ are said to be $\left( {\mathcal E},{\mathcal E}^\ast \right) $-{\it parallel} if $\left( \Delta _1,\Delta _1\cap \overline{\mathcal E},\Delta _1\cap \Sigma ,\Delta _1\cap \overline{{\mathcal E}^\ast }\right) $ and $\left(  \Delta _2,\Delta _2\cap \overline{\mathcal E},\Delta _2\cap \Sigma ,\Delta _2\cap \overline{{\mathcal E}^\ast }\right) $ are isotopic in $\left( H,\overline{\mathcal E},\Sigma ,\partial \overline{{\mathcal E}^\ast }\right) $.
We have the following, similarly to \cite[Proposition~3.6]{Lustig-Moriah}.
\begin{lemma}\label{lem_disk_piece_type}
The set
\begin{equation*}
\bigl\{ \, \text{disk pieces of ${\mathcal D}$ with respect to ${\mathcal E}$}\bigm| \left( {\mathcal D},{\mathcal D}^\ast \right) \in {\mathfrak R}\bigr\} \big/ \text{$\left( {\mathcal E},{\mathcal E}^\ast \right) $-parallel}
\end{equation*}
is finite.
\end{lemma}
\begin{proof}
Note that a disk piece $\Delta $ with respect to ${\mathcal E}$ is a disk, and each lift of it is properly embedded in a $3$-ball in ${\mathcal B}\left( {\mathcal E}\right) $.
It follows from the irreducibility of the $3$-ball that the parallelism class of $\Delta $ is determined by those of the $1$-dimensional objects in $\partial \Delta $.
The present lemma, therefore, almost follows from Lemma~\ref{lem_arc_type}.
It remains to rule out the possibility that infinitely many non-parallel disk pieces are generated by composing arbitrarily many connecting arcs and intersection arcs in the same parallelism classes.
\begin{claim}
At most one connecting arc in $\partial \Delta $ belongs to one $\left( {\mathcal E},{\mathcal E}^\ast \right) $-parallelism class.
\end{claim}
\begin{proof}
Let $\left( {\mathcal D},{\mathcal D}^\ast \right) $ denote the pair in ${\mathfrak R}$ and $Y$ denote the $3$-ball in ${\mathcal B}\left( {\mathcal E}\right) $ such that $\Delta $ is a disk piece of ${\mathcal D}$ with respect to ${\mathcal E}$ in $\pi _H(Y)$.
Assume that distinct connecting arcs $\beta _1$ and $\beta _2$ in $\partial \Delta $ are $\left( {\mathcal E},{\mathcal E}^\ast \right) $-parallel.
There is a subarc $\omega $ of $\partial \overline{\mathcal E}\setminus \overline{{\mathcal E}^\ast }$ connecting $\beta _1$ and $\beta _2$.
Let $\omega ^\circ $ denote the interior of $\omega $.
If $\omega ^\circ $ intersects $\partial \Delta $ then there are more $\left( {\mathcal E},{\mathcal E}^\ast \right) $-parallel connecting arcs.
By rechoosing $\beta _1$ and $\beta _2$ to be innermost, we may suppose that $\omega ^\circ $ is disjoint from $\partial \Delta $.
Assume that $\omega ^\circ $ still intersects $\partial \overline{\mathcal D}$.
There is another disk piece $\Delta '$ of ${\mathcal D}$ with respect to ${\mathcal E}$ in $\pi _H(Y)$ intersecting $\omega ^\circ $.
Let $\widetilde{\Delta }$ denote the lift of $\Delta $ in $Y$, and $\widetilde{\omega }$ denote the lift of $\omega $ in $Y$ connecting the lifts of $\beta _1$ and $\beta _2$.
Note that $\widetilde{\omega }$ and a subarc of $\partial \widetilde{\Delta }$ form a simple, closed curve, which is separating in the sphere $\partial Y$, and that $\partial \Delta $ and $\partial \Delta '$ are disjoint.
It follows that $\omega ^\circ $ and $\partial \Delta '$ have even number of intersection points, and hence $\partial \Delta '$ has distinct two connecting arcs between $\beta _1$ and $\beta _2$.
Note that they are $\left( {\mathcal E},{\mathcal E}^\ast \right) $-parallel as well as $\beta _1$ and $\beta _2$.
By repeating this process, we can rechoose $\beta _1$ and $\beta _2$ to be innermost in such pairs of arcs.
Then $\omega ^\circ $ is disjoint from $\partial \overline{\mathcal D}$, and hence $\omega $ is a wave with respect to ${\mathcal D}$.
Recall that $\omega $ is disjoint from $\overline{{\mathcal E}^\ast }$.
This contradicts the rectangle condition by Lemma~\ref{waves_intersect}.
\end{proof}
\begin{claim}
At most two intersection arcs in $\partial \Delta $ belong to one $\left( {\mathcal E},{\mathcal E}^\ast \right) $-parallelism class.
\end{claim}
\begin{proof}
Assume that distinct three intersection arcs in $\partial \Delta $ are $\left( {\mathcal E},{\mathcal E}^\ast \right) $-parallel.
Let $\widetilde{\Delta }$ be a lift of $\Delta $ in $\widetilde{H}$.
The lifts in $\partial \widetilde{\Delta }$ of at least two of the three arcs are contained in the same disk in $\widetilde{\mathcal E}$.
Similarly to the proof of the previous claim, we obtain a contradiction.
\end{proof}
\noindent
These claims complete the proof of Lemma~\ref{lem_disk_piece_type}.
\end{proof}

Third, we consider the $3$-dimensional objects called thin parts and thick parts.
For a pair $\left( {\mathcal D},{\mathcal D}^\ast \right) $ in ${\mathfrak R}$, a {\it part} of $\widetilde{\mathcal D}$ with respect to $\widetilde{\mathcal E}$ is a component $\widetilde{P}$ of the intersection of a $3$-ball in ${\mathcal B}\left( {\mathcal E}\right) $ and a $3$-ball in ${\mathcal B}\left( {\mathcal D}\right) $.
We also call $\widetilde{P}$ a {\it part} of $B$ with respect to $\widetilde{\mathcal E}$ if $B$ is the $3$-ball in ${\mathcal B}\left( {\mathcal D}\right) $ containing $\widetilde{P}$.
The image $\pi _H\bigl( \widetilde{P}\bigr) $ is denoted by $P$, and called a {\it part} of ${\mathcal D}$ with respect to ${\mathcal E}$.
Note that $\widetilde{P}$ is a $3$-ball, but $P$ is possibly not.
The part $\widetilde{P}$ and the part $P$ are said to be {\it thin} if $P$ is the region of $\left( {\mathcal E},{\mathcal E}^\ast \right) $-parallelism between two disk pieces of ${\mathcal D}$ with respect to ${\mathcal E}$, or {\it thick} otherwise.
For pairs $\left( {\mathcal D}_1,{\mathcal D}_1^\ast \right) $ and $\left( {\mathcal D}_2,{\mathcal D}_2^\ast \right) $ in ${\mathfrak R}$, parts $P_1$ and $P_2$ of ${\mathcal D}_1$ and ${\mathcal D}_2$, respectively, with respect to ${\mathcal E}$ are said to be $\left( {\mathcal E},{\mathcal E}^\ast \right) $-{\it parallel} if $\left( P_1,P_1\cap \overline{\mathcal E},P_1\cap \Sigma ,P_1\cap \overline{{\mathcal E}^\ast }\right) $ and $\left( P_2,P_2\cap \overline{\mathcal E},P_2\cap \Sigma ,P_2\cap \overline{{\mathcal E}^\ast }\right) $ are isotopic in $\left( H,\overline{\mathcal E},\Sigma ,H\cap \overline{{\mathcal E}^\ast }\right) $.

We have the following structure lemma of thin parts.
\begin{lemma}\label{lem_thin_structure}
Let $\left( {\mathcal D},{\mathcal D}^\ast \right) $ be a pair in ${\mathfrak R}$, and $\widetilde{P}$ be a thin part of $\widetilde{\mathcal D}$ with respect to $\widetilde{\mathcal E}$.
Then we have the following:
\begin{itemize}
\item[(1)] Each disk piece of $\widetilde{\mathcal E}$ with respect to $\widetilde{\mathcal D}$ adjacent to $\widetilde{P}$ is a rectangle cobounded by two intersection arcs and two connecting arcs of $\widetilde{\mathcal E}$ with respect to $\widetilde{\mathcal D}$;
\item[(2)] $\widetilde{P}$ intersects exactly two disks in $\widetilde{\mathcal D}$.
\end{itemize}
\end{lemma}
\begin{proof}
The assertion (1) can be seen from the definitions of thinness and $\left( {\mathcal E},{\mathcal E}^\ast \right) $-parallelism for disk pieces.
Also by the thinness, $\widetilde{P}$ is adjacent to exactly two disk pieces of $\widetilde{\mathcal D}$ with respect to $\widetilde{\mathcal E}$.
Assume that they are contained in one disk in $\widetilde{\mathcal D}$.
Then a connecting arc of $\widetilde{\mathcal E}$ with respect to $\widetilde{\mathcal D}$ adjacent to $\widetilde{P}$ maps by $\pi _H$ to a wave with respect to ${\mathcal D}$.
Since the part $\widetilde{P}$ is thin, the wave is disjoint from $\overline{{\mathcal E}^\ast }$.
This contradicts the rectangle condition by Lemma~\ref{waves_intersect}, to conclude the assertion (2).
\end{proof}

We have the following finiteness lemma of thick parts.
\begin{lemma}\label{lem_thick_finiteness}
The set
\begin{equation*}
\bigl\{ \, \text{thick parts of ${\mathcal D}$ with respect to ${\mathcal E}$}\bigm| \left( {\mathcal D},{\mathcal D}^\ast \right) \in {\mathfrak R}\bigr\} \big/ \text{$\left( {\mathcal E},{\mathcal E}^\ast \right) $-parallel}
\end{equation*}
is finite, and there exists the maximum in the set
\begin{equation*}
\bigl\{ \, \text{the number of thick parts of ${\mathcal D}$ with respect to ${\mathcal E}$}\bigm| \left( {\mathcal D},{\mathcal D}^\ast \right) \in {\mathfrak R}\bigr\} .
\end{equation*}
\end{lemma}
\begin{proof}
Note that a thick part $\widetilde{P}$ with respect to $\widetilde{\mathcal E}$ is cut off from a $3$-ball in ${\mathcal B}\left( {\mathcal E}\right) $ by some disk pieces.
The parallelism class of $\pi _H\bigl( \widetilde{P}\bigr) $ is, therefore, almost determined by those of the images of the disk pieces.
To examine more carefully, note that only one stratum of parallel disk pieces possibly cuts off a $3$-ball in ${\mathcal B}\left( {\mathcal E}\right) $ into two thick parts, as illustrated in Figure~\ref{fig_burger}.
Still at most two $\left( {\mathcal E},{\mathcal E}^\ast \right) $-parallelism classes of thick parts correspond to one collection of $\left( {\mathcal E},{\mathcal E}^\ast \right) $-parallelism classes of disk pieces.
Note also that if a thick part has two $\left( {\mathcal E},{\mathcal E}^\ast \right) $-parallel disk pieces, then that part must be the region of parallelism, to contradict the thickness.
Taking these into account, the first assertion follows from Lemma~\ref{lem_disk_piece_type}.
Note also that for a pair $\left( {\mathcal D},{\mathcal D}^\ast \right) $ in ${\mathfrak R}$, two distinct thick parts of ${\mathcal D}$ with respect to ${\mathcal E}$ cannot be $\left( {\mathcal E},{\mathcal E}^\ast \right) $-parallel.
The number of thick parts is, therefore, bounded from above by the finite cardinality in the first assertion, to obtain the second assertion.
\end{proof}

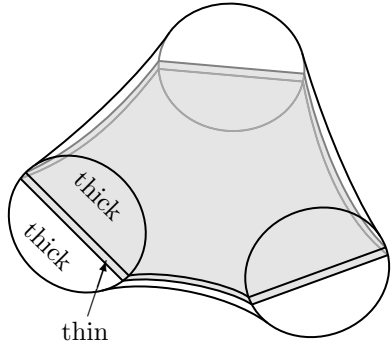
\begin{figure}[ht]
\begin{picture}(150,130)(0,0)
\put(0,0){\includegraphics{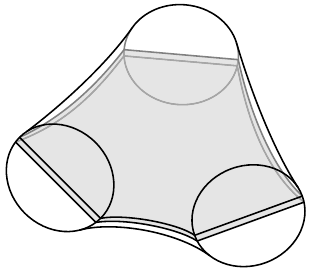}}
\put(27,61){\rotatebox{-45}{thick}}
\put(8,42){\rotatebox{-45}{thick}}
\put(23,1){thin}
\put(33,10){\vector(1,3){7}}
\end{picture}
\caption{Two thick parts separated by a stack of thin parts.}
\label{fig_burger}
\end{figure}

We focus on thin and thick parts lying peripherally in the handlebody.
Let $\left( {\mathcal D},{\mathcal D}^\ast \right) $ be a pair in ${\mathfrak R}$, and $B$ be a $3$-ball in ${\mathcal B}\left( {\mathcal D}\right) $.
Let $\widetilde{P}$ be a thin part of $B$ with respect to $\widetilde{\mathcal E}$, and $\widetilde{\nabla }$ be a disk piece of $\widetilde{\mathcal E}$ with respect to $\widetilde{\mathcal D}$ adjacent to $\widetilde{P}$.
Note that $\widetilde{\nabla }$ is a properly embedded disk in the $3$-ball $B$.
Let ${\mathbf P}$ denote the closure of the complementary component of $\widetilde{\nabla }$ in $B$ such that ${\mathbf P}$ contains $\widetilde{P}$.
Suppose that $B\setminus {\mathbf P}$ intersects all the disks in $\widetilde{\mathcal D}_B$.
Then $\widetilde{P}$ is said to be {\it peripheral}, and ${\mathbf P}$ is called the {\it peripheral component} by $\widetilde{P}$, or a {\it peripheral component} of $B$ with respect to $\widetilde{\mathcal E}$.
We imported this terminology from \cite{Lustig-Moriah}, but modified the definition.
Still similarly to \cite[Lemma~4.5]{Lustig-Moriah}, we have the following structure lemma.
\begin{lemma}\label{lem_peripheral_structure}
Let $\left( {\mathcal D},{\mathcal D}^\ast \right) $ be a pair in ${\mathfrak R}$, and $B$ be a $3$-ball in ${\mathcal B}\left( {\mathcal D}\right) $, and ${\mathbf P}$ be a peripheral component of $B$ with respect to $\widetilde{\mathcal E}$.
Then we have the following:
\begin{itemize}
\item[(1)] Any thin part in ${\mathbf P}$ is peripheral, and the peripheral component by it is contained in ${\mathbf P}$;
\item[(2)] ${\mathbf P}$ intersects exactly two disks $\widetilde{D}_1$ and $\widetilde{D}_2$ in $\widetilde{\mathcal D}_B$;
\item[(3)] ${\mathbf P}\cap \widetilde{D}_1$ and ${\mathbf P}\cap \widetilde{D}_2$ are bigons composed of the same number of disk pieces of $\widetilde{\mathcal D}$ with respect to $\widetilde{\mathcal E}$.
\end{itemize}
\end{lemma}
\noindent
As the assertion (1), there is possibly nesting among peripheral components.
We call a peripheral component {\it maximal} if it is not contained in another.
\begin{proof}[Proof of Lemma~\ref{lem_peripheral_structure}]
The assertion (1) can be seen from the construction of ${\mathbf P}$.
Let $\widetilde{P}$ and $\widetilde{\nabla }$ as above.
By Lemma~\ref{lem_thin_structure}~(2), the thin part $\widetilde{P}$ intersects exactly two disks in $\widetilde{\mathcal D}_B$.
Since ${\mathbf P}$ contains $\widetilde{P}$, the peripheral component ${\mathbf P}$ intersects those two disks.
Assume that ${\mathbf P}$ also intersects another disk $\widetilde{D}_3$ in $\widetilde{\mathcal D}_B$ than the two.
Since $B\setminus {\mathbf P}$ intersects all the disks, $\widetilde{D}_3$ has non-empty intersections with both ${\mathbf P}$ and $B\setminus {\mathbf P}$.
By the construction, the border $\widetilde{\nabla }$, and hence $\widetilde{P}$, intersects $\widetilde{D}_3$.
This is a contradiction, and thus the two disks are $\widetilde{D}_1$ and $\widetilde{D}_2$ in the assertion (2) of the present lemma.
By Lemma~\ref{lem_thin_structure} (1), the disk $\widetilde{\nabla }$ is a rectangle cobounded by two intersection arcs $\widetilde{\alpha }_1$ and $\widetilde{\alpha }_2$ and two connecting arcs $\widetilde{\xi }_1$ and $\widetilde{\xi }_2$ of $\widetilde{\mathcal E}$ with respect to $\widetilde{\mathcal D}$.
The surface ${\mathbf P}\cap \partial B$ is cut off from the $2$-sphere $\partial B$ by the circle $\partial \widetilde{\nabla }$, and hence is a disk.
The surfaces ${\mathbf P}\cap \widetilde{D}_1$ and ${\mathbf P}\cap \widetilde{D}_2$ are cut off from $\widetilde{D}_1$ and $\widetilde{D}_2$, respectively, by $\widetilde{\alpha }_1$ and $\widetilde{\alpha }_2$, and hence are bigons.
Since ${\mathbf P}\cap \partial B$ is disjoint from other disks in $\widetilde{\mathcal D}_B$, the surface ${\mathbf P}\cap \partial \widetilde{H}$ is the closure of the complement of the two bigons in the disk ${\mathbf P}\cap \partial B$, and hence is a rectangle cobounded by $\widetilde{\xi }_1$ and $\widetilde{\xi }_2$ and two subarcs of $\partial \widetilde{D}_1$ and $\partial \widetilde{D}_2$.
Since ${\mathcal E}$ and ${\mathcal D}$ intersect essentially in $H$, the connecting arcs in the rectangle ${\mathbf P}\cap \partial \widetilde{H}$ are parallel, and hence have the same number of endpoints in $\partial \widetilde{D}_1$ and $\partial \widetilde{D}_2$.
They are the endpoints of the same number of intersection arcs in ${\mathbf P}\cap \widetilde{D}_1$ and ${\mathbf P}\cap \widetilde{D}_2$.
They cut off ${\mathbf P}\cap \widetilde{D}_1$ and ${\mathbf P}\cap \widetilde{D}_2$ into the same number of disk pieces, as asserted in (3).
\end{proof}

Based on the above observation, we gather some of the $2$-dimensional objects into what we call peripheral subdisks.
Let $\left( {\mathcal D},{\mathcal D}^\ast \right) $ be a pair in ${\mathfrak R}$, and $\widetilde{D}$ be a disk in $\widetilde{\mathcal D}$.
Let $B_-$ and $B_+$ denote the $3$-balls in ${\mathcal B}\left( {\mathcal D}\right) $ such that $B_-\cap B_+=\widetilde{D}$.
Let ${\mathbf P}_-$ and ${\mathbf P}_+$ be maximal peripheral components of $B_-$ and $B_+$, respectively, with respect to $\widetilde{\mathcal E}$ intersecting $\widetilde{D}$.
Figure~\ref{fig_peripherals} shows their possible relative positions.
We consider the left case, where one of the bigons ${\mathbf P}_-\cap \widetilde{D}$ and ${\mathbf P}_+\cap \widetilde{D}$ contains the other, that is, the smaller bigon is the intersection ${\mathbf P}_-\cap {\mathbf P}_+$.
The image $\pi _H\left( {\mathbf P}_-\cap {\mathbf P}_+\right) $ is called a {\it peripheral subdisk} of ${\mathcal D}$ with respect to ${\mathcal E}$.

\begin{figure}[ht]
\begin{picture}(340,120)(0,0)
\put(0,0){\includegraphics{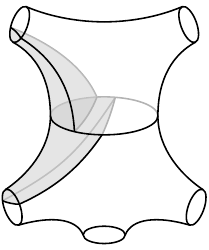}}
\put(24,43){${\mathbf P}_-$}
\put(32,74){${\mathbf P}_+$}
\put(55,25){$B_-$}
\put(55,90){$B_+$}
\put(59,60){$\widetilde{D}$}
\put(120,0){\includegraphics{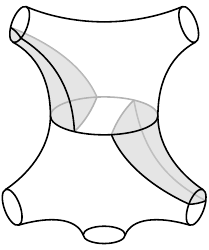}}
\put(240,0){\includegraphics{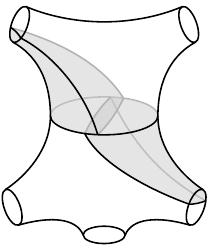}}
\end{picture}
\caption{Relative positions of peripheral components of $B_-$ and $B_+$.}
\label{fig_peripherals}
\end{figure}
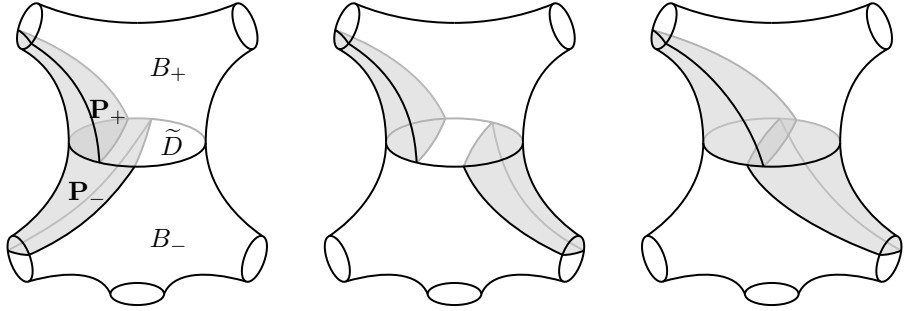

We divide the argument into outside and inside of the peripheral subdisks in the next two paragraphs, respectively, to show the relevant uniform finiteness.
By uniform finiteness, we mean the existence of the maximum in the set
\begin{equation*}
\bigl\{ \, \text{the number of disk pieces of ${\mathcal D}$ with respect to ${\mathcal E}$}\bigm| \left( {\mathcal D},{\mathcal D}^\ast \right) \in \mathfrak{DR}\bigr\} .
\end{equation*}
This and Lemma~\ref{lem_disk_piece_type} imply that we have only finitely many possible constructions for a disk system ${\mathcal D}$ up to isotopy in $H$ so that ${\mathcal D}$ takes part in a pair in $\mathfrak{DR}$.
Whole symmetric arguments hold for disk systems of the other handlebody $H^\ast $, to complete the proof of Theorem~\ref{finiteness}.

The double rectangle condition guarantees the uniform finiteness outside of the peripheral subdisks.
To see this we consider thickness of disk pieces as follows.
A disk piece of ${\mathcal D}$ with respect to ${\mathcal E}$ is said to be {\it thick} if it is adjacent to a thick part of ${\mathcal D}$ with respect to ${\mathcal E}$.
We also define thickness for disk pieces of $\widetilde{\mathcal D}$ with respect to $\widetilde{\mathcal E}$ similarly.
By Lemma~\ref{lem_thick_finiteness}, there exists the maximum, denoted by $T$, in the set
\begin{equation*}
\bigl\{ \, \text{the number of thick disk pieces of ${\mathcal D}$ with respect to ${\mathcal E}$}\bigm| \left( {\mathcal D},{\mathcal D}^\ast \right) \in \mathfrak{DR}\bigr\} .
\end{equation*}
The following lemma implies that $T$ is also a uniform upper bound for the number of disk pieces outside of the peripheral subdisks.
\begin{lemma}\label{lem_central}
For a pair $\left( {\mathcal D},{\mathcal D}^\ast \right) $ in $\mathfrak{DR}$, any disk piece of ${\mathcal D}$ with respect to ${\mathcal E}$ is either thick or contained in a peripheral subdisk.
\end{lemma}
\begin{proof}
We use the notation in Subsection~\ref{criteria_general}.
Note that any disk $E^\ast $ in ${\mathcal E}^\ast $ is disjoint from $\overline{{\mathcal D}^\ast }$ or contains a wave disk with respect to ${\mathcal D}^\ast $.
Let $O^\ast $ denote $E^\ast $ in the former case, or the wave disk in the latter case.
Let $\ell _0$ denote the element in $\{ 1,2,\ldots ,m^\ast \} $ such that the closure of $H^\ast _{\ell _0}$ contains $O^\ast $.
Let $B^\ast _{\ell _0}$ denote the closure of a lift of $H^\ast _{\ell _0}$ in $\widetilde{H^\ast }$.
Let $\widetilde{O^\ast }$ denote the lift of $O^\ast $ in $B^\ast _{\ell _0}$, and $\widetilde{D^\ast }_{v,\nu }$ denote the $\nu $-side lift of $D_v$ in $B^\ast _{\ell _0}$ for $(v,\nu )\in A^\ast _{\ell _0}$.
Let $(V,W)$ denote the partition of $\bigl\{ (v,\nu )\in A^\ast _{\ell _0}\bigm| \widetilde{D^\ast }_{v,\nu }\cap \widetilde{O^\ast }=\emptyset \bigr\} $ such that $\widetilde{O^\ast }$ separates $\bigl\{ \widetilde{D^\ast }_{v,\nu }\bigm| (v,\nu )\in V\bigr\} $ and $\bigl\{ \widetilde{D^\ast }_{w,\mu }\bigm| (w,\mu )\in W\bigr\} $ in $B^\ast _{\ell _0}$.
Now, assume that a disk piece $\Delta $ of ${\mathcal D}$ with respect to ${\mathcal E}$ is neither thick nor contained in any peripheral subdisk.
Let $d$ denote the element in $\{ 1,2,\ldots ,n\} $ such that $D_d$ contains $\Delta $, and $\widetilde{D}_d$ be a lift of $D_d$ in $\widetilde{H}$.
Let $B_\kappa $ denote the $3$-ball in ${\mathcal B}\left( {\mathcal D}\right) $ adjacent to $\widetilde{D}_d$ from the $\kappa $-side for $\kappa \in \{ -,+\} $.
Let $\widetilde{D}_{\kappa ,i,\varepsilon }$ denote the $\varepsilon $-side lift of $D_i$ in $B_\kappa $ for $\kappa \in \{ -,+\} $ and $(i,\varepsilon )\in \varLambda _{d,\kappa }$.
Let $\widetilde{P}_-$ and $\widetilde{P}_+$ denote the parts of $B_-$ and $B_+$ with respect to $\widetilde{\mathcal E}$ adjacent to the lift of $\Delta $ in $\widetilde{D}_d$.
Since $\Delta $ is not thick, $\widetilde{P}_-$ and $\widetilde{P}_+$ are thin parts.
By Lemma~\ref{lem_thin_structure}~(2), there exists a unique member $\left( i_\kappa ,\varepsilon _\kappa \right) $ in $\varLambda _{d,\kappa }$ such that $\widetilde{P}_\kappa $ intersects only $\widetilde{D}_d$ and $\widetilde{D}_{\kappa ,i_\kappa ,\varepsilon _\kappa }$ in $\widetilde{\mathcal D}_{B_\kappa }$, for $\kappa \in \{ -,+\} $.
Since $\Delta $ is not contained in any peripheral subdisk, one of the following holds:
\begin{itemize}
\item Either $\widetilde{P}_-$ or $\widetilde{P}_+$ is not peripheral;
\item Both $\widetilde{P}_-$ and $\widetilde{P}_+$ are peripheral, and the peripheral components by them lie as in the right of Figure~\ref{fig_peripherals}.
\end{itemize}
In either case, $\widetilde{P}_-\cup \widetilde{P}_+$ separates $\bigl( \widetilde{\mathcal D}_{B_-}\cup \widetilde{\mathcal D}_{B_+}\bigr) \setminus \bigl\{ \widetilde{D}_d,\widetilde{D}_{-,i_-,\varepsilon _-},\widetilde{D}_{+,i_+,\varepsilon _+}\bigr\} $ into two or more non-empty subsets in $B_-\cup B_+$.
Let $\left( I_\triangleleft ,I_\triangleright \right) $ be a partition of $\left( \{ -\} \times \varLambda _{d,-}\right) \cup \left( \{ +\} \times \varLambda _{d,+}\right) \setminus \left\{ \left( -,i_-,\varepsilon _-\right) ,\left( +,i_+,\varepsilon _+\right) \right\} $ such that $\widetilde{P}_-\cup \widetilde{P}_+$ separates $\bigl\{ \widetilde{D}_{\kappa ,i,\varepsilon }\bigm| (\kappa ,i,\varepsilon )\in I_\triangleleft \bigr\} $ and $\bigl\{ \widetilde{D}_{\kappa ,i,\varepsilon }\bigm| (\kappa ,i,\varepsilon )\in I_\triangleright \bigr\} $ in $B_-\cup B_+$.
Since the graph ${\mathcal H}_d$ is doubly $2$-connected with respect to $\left( \{ -\} \times \varLambda _{d,-},\{ +\} \times \varLambda _{d,+}\right) $, there exist a member $\left( \kappa _\triangleleft ,i_\triangleleft ,\varepsilon _\triangleleft \right) $ in $I_\triangleleft $ and a member $\left( \kappa _\triangleright ,i_\triangleright ,\varepsilon _\triangleright \right) $ in $I_\triangleright $ that span an edge in ${\mathcal H}_d$.
We divide the argument into two cases:
\begin{itemize}
\item Suppose that $\kappa _\triangleleft $ and $\kappa _\triangleright $ coincide.
Since ${\mathcal G}_{k_{d,\kappa _\triangleleft },\ell _0,i_\triangleleft ,\varepsilon _\triangleleft ,i_\triangleright ,\varepsilon _\triangleright }$ is $2$-connected, there exist a member $\left( v_0,\nu _0\right) $ in $V$ and a member $\left( w_0,\mu _0\right) $ in $W$ that span an edge in ${\mathcal G}_{k_{d,\kappa _\triangleleft },\ell _0,i_\triangleleft ,\varepsilon _\triangleleft ,i_\triangleright ,\varepsilon _\triangleright }$.
By the definition, there exists a rectangle $R$ on $\Sigma $ of type $\left( \left( \partial D_{i_\triangleleft },\varepsilon _\triangleleft \right) ,\left( \partial D_{i_\triangleright },\varepsilon _\triangleright \right) ;\left( \partial D^\ast _{v_0},\nu _0\right) ,\left( \partial D^\ast _{w_0},\mu _0\right) \right) $;
\item Suppose that $\kappa _\triangleleft $ and $\kappa _\triangleright $ are opposite.
Since ${\mathcal H}_{\ell _0,d,i_\triangleleft ,\varepsilon _\triangleleft ,i_\triangleright ,\varepsilon _\triangleright }$ is $2$-connected, there exist a member $\left( v_0,\nu _0\right) $ in $V$ and a member $\left( w_0,\mu _0\right) $ in $W$ that span an edge in ${\mathcal H}_{\ell _0,d,i_\triangleleft ,\varepsilon _\triangleleft ,i_\triangleright ,\varepsilon _\triangleright }$.
By the definition, there exists a rectangle $R$ on $\Sigma $ of type $\left( \left( \partial D_{i_\triangleleft },\varepsilon _\triangleleft \right) ,\partial D_d,\left( \partial D_{i_\triangleright },\varepsilon _\triangleright \right) ;\left( \partial D^\ast _{v_0},\nu _0\right) ,\left( \partial D^\ast _{w_0},\mu _0\right) \right) $, as illustrated in Figure~\ref{fig_intersect_2}.
\end{itemize}
In either case, since $\widetilde{O^\ast }$ separates $\widetilde{D^\ast }_{v_0,\nu _0}$ and $\widetilde{D^\ast }_{w_0,\mu _0}$ in $B^\ast _{\ell _0}$, the rectangle $R$ contains a subarc of $O^\ast \cap \Sigma $ that separates the two edges in $\partial D^\ast _{v_0}$ and $\partial D^\ast _{w_0}$.
It follows that the lift of $O^\ast \cap R$ in $B_-\cup B_+$ connects the disks $\widetilde{D}_{\kappa _\triangleleft ,i_\triangleleft ,\varepsilon _\triangleleft }$ and $\widetilde{D}_{\kappa _\triangleright ,i_\triangleright ,\varepsilon _\triangleright }$, and hence goes across $\widetilde{P}_-\cup \widetilde{P}_+$.
This contradicts that $\widetilde{P}_-$ and $\widetilde{P}_+$ are thin parts.
\end{proof}
\noindent
Lemma~\ref{lem_central} slightly extends as follows.
\begin{lemma}\label{lem_thick_peripheral}
Let $\left( {\mathcal D},{\mathcal D}^\ast \right) $ be a pair in $\mathfrak{DR}$, and $\widetilde{D}$ be a disk in $\widetilde{\mathcal D}$.
Let $B_-$ and $B_+$ denote the $3$-balls in ${\mathcal B}\left( {\mathcal D}\right) $ such that $B_-\cap B_+=\widetilde{D}$.
Let ${\mathbf P}_-$ be a peripheral component of $B_-$ with respect to $\widetilde{\mathcal E}$.
Let $\widetilde{\Delta }$ be any disk piece in ${\mathbf P}_-\cap \widetilde{D}$, and $\widetilde{P}_+$ denote the part of $B_+$ adjacent to $\widetilde{\Delta }$.
Then one of the following holds:
\begin{itemize}
\item The disk piece $\widetilde{\Delta }$ is thick;
\item The part $\widetilde{P}_+$ is thin and peripheral, and the peripheral component by $\widetilde{P}_+$ is adjacent to a subdisk of ${\mathbf P}_-\cap \widetilde{D}$.
\end{itemize}
\end{lemma}
\begin{proof}
Assume the failure of the first assertion.
Then $\widetilde{P}_+$ is thin by the definition, and $\pi _H\bigl( \widetilde{\Delta }\bigr) $ is contained in a peripheral subdisk by Lemma~\ref{lem_central}.
The lift of the peripheral subdisk in $\widetilde{D}$ is the intersection of maximal peripheral components of $B_-$ and $B_+$.
Those components contain ${\mathbf P}_-$ and $\widetilde{P}_+$.
By Lemma~\ref{lem_peripheral_structure}~(1), the thin part $\widetilde{P}_+$ is peripheral, and the peripheral component by $\widetilde{P}_+$ is contained in the maximal one.
Thus the relevant subdisks are well nested, to satisfy the last assertion.
\end{proof}

\begin{figure}[ht]
\begin{picture}(150,180)(0,0)
\put(0,0){\includegraphics{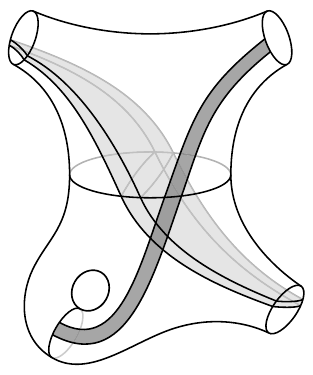}}
\put(127,71){\vector(-2,-1){23}}
\put(129,70){$\pi _H\bigl( \widetilde{P}_-\bigr) $}
\put(46,129){$\pi _H\bigl( \widetilde{P}_+\bigr) $}
\put(18,94){$D_d$}
\put(16,0){$D_{i_\triangleleft }$}
\put(141,164){$D_{i_\triangleright }$}
\put(45,56){$\partial D^\ast _{v_0}$}
\put(72,34){$\partial D^\ast _{w_0}$}
\put(129,122){\vector(-1,1){18}}
\put(130,115){$R$}
\end{picture}
\caption{The rectangle $R$ in a possible case where both $\widetilde{P}_-$ and $\widetilde{P}_+$ are peripheral.}
\label{fig_intersect_2}
\end{figure}
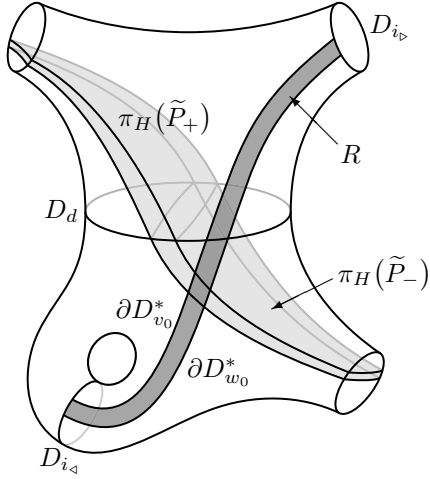

It remains to show the uniform finiteness inside of the peripheral subdisks.
We can divide it into the following two lemmas.
\begin{lemma}\label{lem_peripheral_number}
There exists the maximum in the set
\begin{equation*}
\bigl\{ \, \text{the number of peripheral subdisks of ${\mathcal D}$ with respect to ${\mathcal E}$}\bigm| \left( {\mathcal D},{\mathcal D}^\ast \right) \in \mathfrak{DR}\bigr\} .
\end{equation*}
\end{lemma}
\begin{proof}
Let $\left( {\mathcal D},{\mathcal D}^\ast \right) $ be a pair in $\mathfrak{DR}$.
By Lemma~\ref{lem_central}, the number of peripheral subdisks of ${\mathcal D}$ with respect to ${\mathcal E}$ is less than or equal to the number of complementary components in $\overline{\mathcal D}$ of the thick disk pieces.
Recall that ${\mathcal D}$ consists of at most $3g-3$ disks, and they contains at most $T$ thick disk pieces.
By Lemma~\ref{lem_disk_piece_type}, there exists the maximum, denoted by $c$, in the set
\begin{equation*}
\left\{ \, 
\begin{minipage}{130pt}
the number of complementary components in $D$ of $\Delta $
\end{minipage}
\; \middle| \; 
\begin{minipage}{195pt}
$\left( {\mathcal D},{\mathcal D}^\ast \right) \in {\mathfrak R}$,\\
$D\in {\mathcal D}$,\\
$\Delta $ is a disk piece of ${\mathcal D}$ with respect to ${\mathcal E}$ in $D$
\end{minipage}
\right\} .
\end{equation*}
The number at issue is, therefore, coarsely bounded from above by $3g-3+T(c-1)$.
\end{proof}
\begin{lemma}\label{lem_peripheral_area}
There exists the maximum in the set
\begin{equation*}
\left\{ \, \Area \left( {\mathbf \Delta }\right) \; \middle| \; 
\begin{minipage}{210pt}
$\left( {\mathcal D},{\mathcal D}^\ast \right) \in \mathfrak{DR}$,\\
${\mathbf \Delta }$ is a peripheral subdisk of ${\mathcal D}$ with respect to ${\mathcal E}$
\end{minipage}
\right\} .
\end{equation*}
\end{lemma}
\noindent
Here, $\Area \left( {\mathbf \Delta }\right) $ denotes the number of disk pieces in ${\mathbf \Delta }$.
We use this symbol not only for a peripheral subdisk ${\mathbf \Delta }$, but also for a general disk composed of some disk pieces of ${\mathcal D}$ with respect to ${\mathcal E}$, or of $\widetilde{\mathcal D}$ with respect to $\widetilde{\mathcal E}$.
\begin{proof}[Proof of Lemma~\ref{lem_peripheral_area}]
This proof is based on the idea of ``disk pushing procedure'' by Lustig--Moriah \cite[Section~5]{Lustig-Moriah}.
Let $\left( {\mathcal D},{\mathcal D}^\ast \right) $ be a pair in $\mathfrak{DR}$, and ${\mathbf \Delta }$ be a peripheral subdisk of ${\mathcal D}$ with respect to ${\mathcal E}$.
We construct a sequence of disks as illustrated schematically in Figure~\ref{fig_pushing}, and derive from it a uniform upper bound for $\Area \left( {\mathbf \Delta }\right) $ as follows.
Let $\widetilde{\mathbf \Delta }_1$ be a lift of ${\mathbf \Delta }$ in $\widetilde{H}$, and $\widetilde{D}_1$ denote the disk in $\widetilde{\mathcal D}$ containing $\widetilde{\mathbf \Delta }_1$.
Let $B_0$ and $B_1$ denote the $3$-balls in ${\mathcal B}\left( {\mathcal D}\right) $ such that $B_0\cap B_1=\widetilde{D}_1$.
Let ${\mathbf P}_0$ and ${\mathbf P}_1$ denote the maximal peripheral components of $B_0$ and $B_1$, respectively, with respect to $\widetilde{\mathcal E}$ containing $\widetilde{\mathbf \Delta }_1$.
We may suppose that $\widetilde{\mathbf \Delta }_1={\mathbf P}_1\cap \widetilde{D}_1$ without loss of generality.
For an inductive construction, assume that $\widetilde{\mathbf \Delta }_\iota $, $\widetilde{D}_\iota $, $B_\iota $, ${\mathbf P}_\iota $ have been defined for a positive integer $\iota $.
If ${\mathbf P}_\iota $ is empty, then set $\widetilde{\mathbf \Delta }_{\iota +1}$, $\widetilde{D}_{\iota +1}$, $B_{\iota +1}$, ${\mathbf P}_{\iota +1}$ to be empty.
Assume ${\mathbf P}_\iota $ to be non-empty.
By Lemma~\ref{lem_peripheral_structure}~(2), it intersects only $\widetilde{D}_\iota $ and another disk $\widetilde{D}_{\iota +1}$ in $\widetilde{\mathcal D}_{B_\iota }$.
Let $B_{\iota +1}$ denote the $3$-ball in ${\mathcal B}\left( {\mathcal D}\right) $ adjacent to $\widetilde{D}_{\iota +1}$ other than $B_\iota $, and $\widetilde{\mathbf \Delta }'_{\iota +1}$ denote the disk ${\mathbf P}_\iota \cap \widetilde{D}_{\iota +1}$.
If there exist peripheral components of $B_{\iota +1}$ with respect to $\widetilde{\mathcal E}$ that intersect $\widetilde{D}_{\iota +1}$ in $\widetilde{\mathbf \Delta }'_{\iota +1}$, then let ${\mathbf P}_{\iota +1}$ be one of those peripheral components such that $\Area \bigl( {\mathbf P}_{\iota +1}\cap \widetilde{D}_{\iota +1}\bigr) $ is maximal.
If not exist, then set ${\mathbf P}_{\iota +1}$ to be empty.
Let $\widetilde{\mathbf \Delta }_{\iota +1}$ denote the set ${\mathbf P}_{\iota +1}\cap \widetilde{D}_{\iota +1}$.
In this way, we obtain the sequence $\bigl( \widetilde{\mathbf \Delta }_\iota \bigr) _{\iota \in {\mathbb N}}$ together with $\bigl( \widetilde{D}_\iota \bigr) _{\iota \in {\mathbb N}}$, $\left( B_\iota \right) _{\iota \in {\mathbb N}}$, $\left( {\mathbf P}_\iota \right) _{\iota \in {\mathbb N}}$ and $\bigl( \widetilde{\mathbf \Delta }'_{\iota +1}\bigr) _{\iota \in {\mathbb N}}$.
Note that $\Area \bigl( \widetilde{\mathbf \Delta }'_{\iota +1}\bigr) \geq \Area \bigl( \widetilde{\mathbf \Delta }_{\iota +1}\bigr) $ since $\widetilde{\mathbf \Delta }'_{\iota +1}\supset \widetilde{\mathbf \Delta }_{\iota +1}$, and that $\Area \bigl( \widetilde{\mathbf \Delta }_\iota \bigr) =\Area \bigl( \widetilde{\mathbf \Delta }'_{\iota +1}\bigr) $ by Lemma~\ref{lem_peripheral_structure}~(3).
The sequence $\bigl( \Area \bigl( \widetilde{\mathbf \Delta }_\iota \bigr) \bigr) _{\iota \in {\mathbb N}}$ is, therefore, monotonically decreasing.
We analyze how it decreases in the following claims.
\begin{claim}
There exists a positive integer such that $\Area \bigl( \widetilde{\mathbf \Delta }_\iota \bigr) =0$ for every greater integer $\iota $.
\end{claim}
\begin{proof}
Since $\bigl( \Area \bigl( \widetilde{\mathbf \Delta }_\iota \bigr) \bigr) _{\iota \in {\mathbb N}}$ is a decreasing sequence of non-negative integers, it stays constant after $\iota $ exceeds some integer $L$.
It remains to show that the constant area is zero.
If not, $\left( {\mathbf P}_\iota \right) _{\iota >L}$ is a sequence of non-empty peripheral components with constant cross-section area.
Then ${\mathbf P}_{\iota +1}$ is uniquely determined by ${\mathbf P}_\iota $ to be adjacent to the same subdisk in $\widetilde{D}_{\iota +1}$, for each integer $\iota $ greater than $L$.
Note that there are only finitely many parts of ${\mathcal D}$ with respect to ${\mathcal E}$, and hence finitely many images of peripheral components.
It follows that the sequence $\left( \pi _H\left( {\mathbf P}_\iota \right) \right) _{\iota >L}$ is cyclic.
Note that ${\mathbf P}_\iota $ is cut off from $B_\iota $ by a rectangular disk piece of $\widetilde{\mathcal E}$ with respect to $\widetilde{\mathcal D}$.
The images of such disk pieces form an annulus in a disk in ${\mathcal E}$, to contradict that ${\mathcal E}$ and ${\mathcal D}$ intersect essentially.
\end{proof}
\noindent
This says that the length of the sequence $\bigl( \widetilde{\mathbf \Delta }_\iota \bigr) _{\iota \in {\mathbb N}}$ is essentially finite, but does not mean the uniform boundedness of that length.
Instead we consider the length of $\bigl( \Area \bigl( \widetilde{\mathbf \Delta }_\iota \bigr) \bigr) _{\iota \in {\mathbb N}}$ after reducing repetitions of the same areas.
Let $N$ denote the set $\bigl\{ \iota \in {\mathbb N}\bigm| \Area \bigl( \widetilde{\mathbf \Delta }_\iota \bigr) >\Area \bigl( \widetilde{\mathbf \Delta }_{\iota +1}\bigr) \bigr\} $.
\begin{claim}
The cardinality of $N$ is less than or equal to the maximum number $T$ of thick disk pieces.
\end{claim}
\begin{proof}
Suppose that $\iota $ belongs to $N$, that is, $\Area \bigl( \widetilde{\mathbf \Delta }_\iota \bigr) >\Area \bigl( \widetilde{\mathbf \Delta }_{\iota +1}\bigr) $.
Since $\Area \bigl( \widetilde{\mathbf \Delta }_\iota \bigr) $ is positive, ${\mathbf P}_\iota $ is not empty.
Let $\widetilde{P}_\iota $ denote the peripheral thin part such that ${\mathbf P}_\iota $ is the peripheral component by $\widetilde{P}_\iota $.
Let $\widetilde{\Delta }_{\iota +1}$ denote the disk piece in $\widetilde{D}_{\iota +1}$ adjacent to $\widetilde{P}_\iota $.
Assume that $\widetilde{\Delta }_{\iota +1}$ is not thick.
Then by Lemma~\ref{lem_thick_peripheral}, there is a peripheral component of $B_{\iota +1}$ adjacent to $\widetilde{\mathbf \Delta }'_{\iota +1}$.
By the construction, that peripheral component must be ${\mathbf P}_{\iota +1}$.
It follows that $\Area \bigl( \widetilde{\mathbf \Delta }_\iota \bigr) =\Area \bigl( \widetilde{\mathbf \Delta }'_{\iota +1}\bigr) =\Area \bigl( \widetilde{\mathbf \Delta }_{\iota +1}\bigr) $, a contradiction.
Thus $\widetilde{\Delta }_{\iota +1}$, and hence $\pi _H\bigl( \widetilde{\Delta }_{\iota +1}\bigr) $, is a thick disk piece.
Note that the sequence $\bigl( \pi _H\bigl( \widetilde{\Delta }_{\iota +1}\bigr) \bigr) _{\iota \in N}$ does not recur the same term, since $\bigl( \Area \bigl( \widetilde{\mathbf \Delta }_\iota \bigr) \bigr) _{\iota \in N}$ is strictly decreasing.
This gives the claimed bound.
\end{proof}
\noindent
Next we estimate how sharply $\Area \bigl( \widetilde{\mathbf \Delta }_\iota \bigr) $ decreases at each step of $\iota $.
\begin{claim}
There exists a constant $r$, independent from the choice of $\left( {\mathcal D},{\mathcal D}^\ast \right) $ in $\mathfrak{DR}$, such that $\Area \bigl( \widetilde{\mathbf \Delta }_\iota \bigr) <r\bigl( \Area \bigl( \widetilde{\mathbf \Delta }_{\iota +1}\bigr) +1\bigr) $ for every positive integer $\iota $.
\end{claim}
\begin{proof}
We may suppose that $\iota $ belongs to $N$.
There are some peripheral components of $B_{\iota +1}$ with respect to $\widetilde{\mathcal E}$ whose intersections with $\widetilde{D}_{\iota +1}$ are proper subdisks of $\widetilde{\mathbf \Delta }'_{\iota +1}$.
Let $\widetilde{\mathbf \Delta }_{\iota +1}^\flat $ denote the closure of the exterior in $\widetilde{\mathbf \Delta }'_{\iota +1}$ of those peripheral components.
It follows from Lemma~\ref{lem_thick_peripheral} that $\widetilde{\mathbf \Delta }_{\iota +1}^\flat $ is composed of thick disk pieces, which are no more than $T$.
As mentioned in the proof of Lemma~\ref{lem_peripheral_number}, there exists a uniform upper bound $c$ for the number of complementary components of any disk piece.
The complementary components of $\widetilde{\mathbf \Delta }_{\iota +1}^\flat $ in $\widetilde{\mathbf \Delta }'_{\iota +1}$ are no more than $T(c-1)$, and have areas at most $\Area \bigl( \widetilde{\mathbf \Delta }_{\iota +1}\bigr) $.
Thus $\Area \bigl( \widetilde{\mathbf \Delta }'_{\iota +1}\bigr) $, and hence $\Area \bigl( \widetilde{\mathbf \Delta }_\iota \bigr) $, is bounded from above by $T+T(c-1)\Area \bigl( \widetilde{\mathbf \Delta }_{\iota +1}\bigr) $, and hence by $T(c+1)\bigl( \Area \bigl( \widetilde{\mathbf \Delta }_{\iota +1}\bigr) +1\bigr) $.
\end{proof}
\noindent
These show that $\bigl( \Area \bigl( \widetilde{\mathbf \Delta }_\iota \bigr) \bigr) _{\iota \in N}$ is bounded from above by a geometric sequence, and that $\Area \left( {\mathbf \Delta }\right) =\Area \bigl( \widetilde{\mathbf \Delta }_1\bigr) <r^T+r^{T-1}+r^{T-2}+\cdots +r+1$.
\end{proof}

\begin{figure}[ht]
\begin{picture}(320,150)(0,0)
\put(0,0){\includegraphics{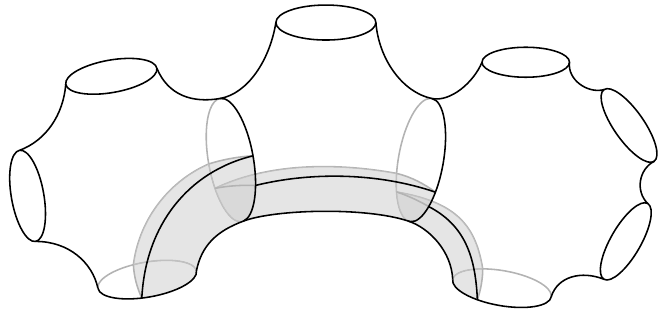}}
\put(25,5){$\widetilde{D}_{\iota -1}$}
\put(83,0){$\widetilde{\mathbf \Delta }_{\iota -1}$}
\put(48,71){$B_{\iota -1}$}
\put(82,39){${\mathbf P}_{\iota -1}$}
\put(98,108){$\widetilde{D}_\iota $}
\put(120,38){\vector(-1,3){5}}
\put(113,25){$\widetilde{\mathbf \Delta }'_\iota \supset \widetilde{\mathbf \Delta }_\iota $}
\put(152,98){$B_\iota $}
\put(152,54){${\mathbf P}_\iota $}
\put(199,109){$\widetilde{D}_{\iota +1}$}
\put(187,14){\vector(1,3){12}}
\put(138,3){$\widetilde{\mathbf \Delta }_{\iota +1}\subset \widetilde{\mathbf \Delta }'_{\iota +1}$}
\put(253,67){$B_{\iota +1}$}
\put(230,41){${\mathbf P}_{\iota +1}$}
\put(267,1){$\widetilde{D}_{\iota +2}$}
\put(212,0){$\widetilde{\mathbf \Delta }'_{\iota +2}$}
\end{picture}
\caption{Construction of a sequence of disks.}
\label{fig_pushing}
\end{figure}
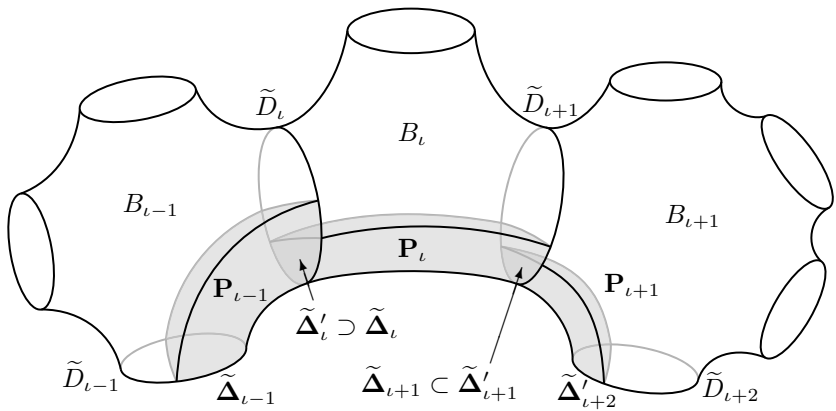


\begin{thebibliography}{99}
\bibitem{Casson-Gordon} A. J. Casson and C. McA. Gordon, {\it Reducing Heegaard splittings}, Topology Appl. {\bf 27} (1987), no. 3, 275--283.
\bibitem{Farb-Margalit} B. Farb and D. Margalit, A primer on mapping class groups, Princeton Mathematical Series, {\bf 49}. Princeton University Press, Princeton, 2012.
\bibitem{Fomenko-Matveev} A. T. Fomenko and S. V. Matveev, Algorithmic and computer methods for three-manifolds, Mathematics and its Applications, {\bf 425}. Kluwer Academic Publishers, Dordrecht, 1997.
\bibitem{Haken} W. Haken, {\it Some results on surfaces in 3-manifolds}, Studies in Modern Topology, 39--98  Math. Assoc. Amer., 1968.
\bibitem{Hartshorn} K. Hartshorn, {\it Heegaard splittings of Haken manifolds have bounded distance}, Pacific J. Math. {\bf 204} (2002), no. 1, 61--75.
\bibitem{Hempel} J. Hempel, {\it 3-manifolds as viewed from the curve complex}, Topology {\bf 40} (2001), no. 3, 631--657.
\bibitem{Iguchi-Koda} D. Iguchi and Y. Koda, {\it Twisted book decompositions and the Goeritz groups}, Topology Appl. {\bf 272} (2020), 107064, 15 pp.
\bibitem{Johnson10} J. Johnson, {\it Mapping class groups of medium distance Heegaard splittings}, Proc. Amer. Math. Soc. {\bf 138} (2010), no. 12, 4529--4535.
\bibitem{Johnson(11)} J. Johnson, {\it Heegaard splittings and open books}, arXiv:1110.2142.
\bibitem{Lackenby} M. Lackenby, {\it The Heegaard genus of amalgamated 3-manifolds}, Geom. Dedicata {\bf 109} (2004), 139--145.
\bibitem{Li07} T. Li, {\it On the Heegaard splittings of amalgamated 3-manifolds}, Workshop on Heegaard Splittings, 157--190, Geom. Topol. Monogr., {\bf 12}, Geom. Topol. Publ., Coventry, 2007.
\bibitem{Li10} T. Li, {\it Heegaard surfaces and the distance of amalgamation}, Geom. Topol. {\bf 14} (2010), no. 4, 1871--1919.
\bibitem{Lustig-Moriah} M. Lustig and Y. Moriah, {\it A finiteness result for Heegaard splittings}, Topology {\bf 43} (2004), no. 5, 1165--1182.
\bibitem{Serre} J.-P. Serre, {\it Rigidit\'e du foncteur de Jacobi d'\'echelon $n\geq 3$}, Appendice \`a l'expose 17, S\'eminaire Henri Cartan 13e annee, 1960/61.
\bibitem{Souto} J. Souto, {\it Geometry, Heegaard splittings and rank of the fundamental group of hyperbolic 3-manifolds}, Workshop on Heegaard Splittings, 351--399, Geom. Topol. Monogr., {\bf 12}, Geom. Topol. Publ., Coventry, 2007.
\bibitem{Zou} Y. Zou, {\it Finiteness of mapping class groups and Heegaard distance}, arXiv:2303.02849.
\end{thebibliography}
\end{document}